\documentclass{amsart}
\usepackage[margin=1.25in]{geometry}
\usepackage{amsmath,amsthm,amssymb,amsaddr}

\usepackage{bm,bbm}
\usepackage{enumerate}
\usepackage{amsmath,amssymb}
\usepackage{stmaryrd}
\usepackage{booktabs}
\usepackage{tabularx}
\usepackage{multirow}
\usepackage{xcolor,graphicx}

\newtheorem{rem}{Remark}
\newtheorem{prop}{Proposition}
\newtheorem{lem}{Lemma}

\newcommand{\E}{\mathcal{E}}
\newcommand{\T}{\mathcal{T}}
\newcommand{\R}{\mathbb{R}}
\newcommand{\wt}{\widetilde}
\newcommand{\half}{\frac{1}{2}}

\begin{document}

\title[Analysis and entropy stability of Line-DG]
  {Analysis and entropy stability of the line-based discontinuous Galerkin
       method}
\author{Will Pazner}
\address{Center for Applied Scientific Computing, Lawrence Livermore National
        Laboratory}
\author{Per-Olof Persson}
\address{Department of Mathematics, University of California, Berkeley}

\maketitle

\begin{abstract}
We develop a discretely entropy-stable line-based discontinuous Galerkin method
for hyperbolic conservation laws based on a flux differencing technique. By
using standard entropy-stable and entropy-conservative numerical flux functions,
this method guarantees that the discrete integral of the entropy is
non-increasing. This nonlinear entropy stability property is important for the
robustness of the method, in particular when applied to problems with
discontinuous solutions or when the mesh is under-resolved. This line-based
method is significantly less computationally expensive than a standard DG
method. Numerical results are shown demonstrating the effectiveness of the
method on a variety of test cases, including Burgers' equation and the Euler
equations, in one, two, and three spatial dimensions.
\end{abstract}

\section{Introduction}

High-order numerical methods for the solution of partial differential equations
have seen success in a wide range of application areas \cite{Wang2013}. In
particular, the discontinuous Galerkin (DG) method \cite{Cockburn1991,Reed1973},
an arbitrary-order finite element method suitable for use on unstructured
geometries, possesses many desirable properties, making it well-suited for a
large number of applications. Variants of the DG method, such as the
discontinuous Galerkin spectral element method (DG-SEM)
\cite{Black2000,Rasetarinera2001,Kopriva1996} and the line-based discontinuous
Galerkin method (Line-DG) \cite{Persson2012,Persson2013} have been introduced in
order to retain the attractive properties of the DG method while reducing its
computational cost.

Of particular interest are the stability properties of these methods. It has
been shown that the standard discontinuous Galerkin method satisfies a cell
entropy inequality in the scalar and symmetric system case, leading to $L^2$
stability \cite{Jiang1994,Hou2006}. However, these results do not immediately
translate to general systems of conservation laws. Additionally, these results
rely on the use of exactly integrated DG methods, which may be impractical or
prohibitively expensive due to nonlinearities in the fluxes. To maintain
stability in the general nonlinear case, a variety of methods have been
proposed, including limiters \cite{Zhang2010b,Cockburn2001} and artificial
viscosity \cite{Persson2006,Zingan2013}, however these methods can result in
reduced order of accuracy, and often require parameter tuning. Recently,
discretely entropy-stable DG and DG-SEM methods have been developed based on a
technique known as flux differencing
\cite{Chan2018,Chen2017,Carpenter2014,Fisher2013,Gassner2016,Gassner2013,%
Parsani2016}. These methods are based on the entropy-conservative and
entropy-stable fluxes developed by Tadmor
\cite{Tadmor1987,Tadmor2003,Osher1988}, which have been used in the context of
finite volume methods.

In this work, we extend the flux differencing methodology to the line-based DG
(Line-DG) methods developed in \cite{Persson2012,Persson2013}. These methods are
closely related both to standard DG methods and to DG-SEM methods, and are based
on solving a sequence of one-dimensional Galerkin problems along lines of nodes
within a tensor-product element. We modify the Line-DG method by introducing
entropy-stable and entropy-conservative flux functions, combined with
appropriate projection operators required to ensure discrete entropy stability.
Then we show that this modified method satisfies an entropy inequality
consistent with the quadrature chosen for the scheme. This discrete entropy
stability property is shown to be important for the robustness of the scheme.
For instance, for Burgers' equation, entropy stability implies $L^2$ stability.
For the Euler equations, we additionally must require that the density and
pressure remain positive. We remark that although the method remains stable, the
numerical solution may still develop strong oscillations in the vicinity of a
discontinuity, suggesting the utility of other shock-capturing techniques for
problems with strong shocks. However, this increased robustness could prove to
be particularly important for under-resolved turbulent flows, for which standard
methods have been observed to be unstable \cite{Winters2018,Moura2017}. The
structure of the paper is as follows. In Section \ref{sec:eqns-discretization}
we describe the governing equations and define the Line-DG method. We then
modify the standard Line-DG scheme to achieve discrete entropy stability, and
analyze the accuracy of the resulting method. In Section
\ref{sec:implementation} we discuss implementation details and computational
efficiency. In Section \ref{sec:results} we provide a range of numerical
experiments, including Burgers' equation and the Euler equations of gas
dynamics, in one, two, and three spatial dimensions. We end with concluding
remarks in Section \ref{sec:conclusion}.

\section{Discretization and equations} \label{sec:eqns-discretization}

\subsection{Governing equations and entropy analysis}

We consider a system of hyperbolic conservation laws in $d$ dimensions in a
spatial domain $\Omega \subseteq \R^d$,
\begin{equation} \label{eq:cons-law}
  \frac{\partial \bm u}{\partial t} + \nabla \cdot \bm f = 0.
\end{equation}
The number of solution components is denoted $n_c$, and so the solution $\bm u$
is a function $\bm u (\bm x, t) : \R^d \times \R \to \R^{n_c}$, and the flux
function is written $\bm f(\bm u) : \R^{n_c} \to \R^{n_c \times d}$.

A convex function $U (\bm u) : \R^{n_c} \to \R$ is called an \textit{entropy
function} if there exist \textit{entropy fluxes} $F_j : \R^{n_c} \to \R$, $1
\leq j \leq d$, such that
\begin{equation}
  U'(\bm u) \bm f'_j (\bm u) = F'_j (\bm u),
\end{equation}
where the derivatives are taken with respect to the state variables $\bm u$. In
regions where $\bm u$ is smooth, the entropy satisfies the related equation
\begin{equation} \label{eq:entropy-eq}
  \frac{\partial U}{\partial t} + \nabla \cdot \bm F = 0.
\end{equation}
However, hyperbolic conservation laws admit solutions with discontinuities, for
which the physically-relevant solutions must dissipate entropy. Thus, the
\textit{entropy solution} $\bm u$ to equation \eqref{eq:cons-law} satisfies
\begin{equation} \label{eq:entropy-inequality}
  \frac{\partial U}{\partial t} + \nabla \cdot \bm F \leq 0
\end{equation}
for all entropy functions $U$. Assuming periodic or compactly supported boundary
conditions, and integrating \eqref{eq:entropy-inequality} over $\Omega$, we
obtain
\begin{equation} \label{eq:entropy-decreasing}
  \frac{d}{dt} \int_\Omega U \, d\bm x \leq 0,
\end{equation}
and thus conclude that the total entropy is monotonically non-increasing in
time. The inequality \eqref{eq:entropy-decreasing} can be seen as a non-linear
analogue to standard $L^2$ stability. Furthermore, if the entropy function $U$
is uniformly convex, then entropy stability can be used to guarantee $L^2$
stability, thus motivating the development of numerical schemes that satisfy a
discrete entropy stability property.

\subsubsection{Entropy variables}
We define the \textit{entropy variables} by
\begin{equation}
  \bm v = U'(\bm u).
\end{equation}
If $U$ is uniformly convex, then the mapping $\bm u \mapsto \bm v$ is
invertible, and is considered as a change of variables. Defining $\bm g(\bm v)
= \bm f(\bm u(\bm v))$, we obtain a system of hyperbolic conservation laws
equivalent to \eqref{eq:cons-law},
\begin{equation}
  \bm u'(\bm v) \frac{\partial\bm v}{\partial t} + \nabla \cdot \bm g = 0.
\end{equation}
Convexity of the entropy function $U$ implies symmetry of $\bm g'(\bm v)$, and
so there exist functions $\psi_j(\bm v)$, called \textit{flux potentials},
satisfying
\begin{equation}
  \psi_j'(\bm v) = \bm g_j(\bm v).
\end{equation}
One can verify that $\psi_j$ is given by
\begin{equation}
  \psi_j = \bm g_j(\bm v)^T \bm v - F_j(\bm u(\bm v)).
\end{equation}

\subsubsection{Discrete entropy stability and two-point numerical fluxes}

Given that the governing conservation law \eqref{eq:cons-law} satisfies the
entropy inequality \eqref{eq:entropy-decreasing}, it is desirable for the
numerical method to satisfy a corresponding discrete analogue, which can then be
used to attain nonlinear stability for the numerical scheme. In the context
of finite volume methods, Tadmor demonstrated how this can be achieved through
the judicious choice of numerical flux functions \cite{Tadmor1987}. Writing a
finite volume method for
\eqref{eq:cons-law} as
\begin{equation} \label{eq:fv-scheme}
   \frac{d \bm u_i}{dt} = -\frac{1}{\Delta x}
      \left( \bm f_{i+1/2} - \bm f_{i-1/2} \right),
\end{equation}
and multiplying on the left by the entropy variables $\bm v_i^T$, we obtain
\begin{equation} \label{eq:fv-entropy}
   \frac{d U_i}{dt} = - \frac{1}{\Delta x}
      \bm v_i^T \left( \bm f_{i+1/2} - \bm f_{i-1/2} \right).
\end{equation}
Thus, the scheme \eqref{eq:fv-scheme} can be seen to be entropy conservative if
the above equation can be rewritten as
\begin{equation} \label{eq:fv-entropy-cons}
   \frac{d U_i}{dt} = - \frac{1}{\Delta x}
   \left( F_{i+1/2} - F_{i-1/2} \right),
\end{equation}
which is a finite volume discretization of the entropy equation
\eqref{eq:entropy-eq}. Tadmor demonstrated that if $\bm f_{i+1/2}$
satisfies the so-called \textit{shuffle condition},
\begin{equation}
   \left( \bm v_{i+1} - \bm v_i \right)^T \bm f_{i+1/2} = \psi_{i+1} - \psi_i,
\end{equation}
then \eqref{eq:fv-entropy-cons} is satisfied and \eqref{eq:fv-scheme} is entropy
conservative with
\begin{equation}
   F_{i+1/2} = \frac{1}{2} \left( \bm v_{i+1} + \bm v_{i} \right)^T\bm f_{i+1/2}
      - \frac{1}{2}\left(\psi_{i+1} + \psi_i \right).
\end{equation}
The addition of dissipative terms will then lead to an entropy stable scheme.
In the work of Fisher, it was shown how to extend this construction to the
context of spectral element methods through a flux differencing technique
\cite{Carpenter2014,Fisher2013}.

Working within a similar framework, we introduce the notation required to
obtain discrete entropy stability for the line-based DG method. We define
\textit{entropy-conservative} and \textit{entropy-stable} numerical flux
functions. An \textit{entropy-conservative} two-point numerical flux is a
function $\bm f_{i,S}$, $1 \leq i \leq d$ that satisfies the following
properties:
\begin{enumerate}
\item Consistency: $\bm f_{i,S}(\bm u, \bm u) = \bm f_i(\bm u)$.
\item Symmetry: $\bm f_{i,S}(\bm u_L, \bm u_R) = \bm f_{i,S}(\bm u_R, \bm u_L)$.
\item Entropy conservation:
  $(\bm v_R - \bm v_L)^T \bm f_{i,S}(\bm u_L, \bm u_R) = \psi_{i,R}-\psi_{i,L}$.
\end{enumerate}
Additionally, we introduce an \textit{entropy-stable} two-point numerical flux
function $\widehat{\bm f}_S$ that satisfies:
\begin{enumerate}
\item Consistency: $\widehat{\bm f}_S(\bm u, \bm u) = \bm f(\bm u)$.
\item Symmetry: $\widehat{\bm f}_S(\bm u_L, \bm u_R)
  = \widehat{\bm f}_S(\bm u_R, \bm u_L)$.
\item Entropy stability:
  $(\bm v_R - \bm v_L)^T \widehat{\bm f}_S(\bm u_L, \bm u_R) \cdot \bm n
    \leq \left(\bm\psi_R - \bm\psi_L\right)\cdot \bm n$.
\end{enumerate}
We point out that in the entropy-stable case, the equality in the third property
has been replaced by an inequality. The states $\bm u_L$ and
$\bm u_R$ can either be traces of the states at an element interface evaluated
from within neighboring elements, or simply the state variables evaluated at
different points within a single element, as will be described in Section
\ref{sec:modified}.

\subsection{Euler equations and entropy variables}

As a particularly important example of governing equations, we consider the
Euler equations of gas dynamics in $d$ spatial dimensions, written in
conservative form
\begin{equation}
  \frac{\partial \bm u}{\partial t} + \nabla \cdot \bm f(\bm u) = 0,
\end{equation}
where $\bm u = (\rho, \rho \bm w, \rho E)$ is the vector of conserved variables:
$\rho$ is the density, $\rho \bm w$ is the momentum ($\bm w$ denotes
the fluid velocity), and $\rho E$ is the total energy. The flux function is
given by
\begin{equation}
  \bm f (\bm u) = \left( \begin{array}{c}
    \rho \bm w \\
    \rho \bm w \otimes \bm w^T + p I \\
    \rho H \bm w
  \end{array}\right).
\end{equation}
Here $I$ is the $d \times d$ identity matrix, $p$ is the pressure, and
$H = E + p/\rho$ is the stagnation enthalpy. The pressure is defined through the
equation of state
\begin{equation}
  p = (\gamma - 1) \rho (E - \|\bm w\|^2/2),
\end{equation}
where the constant $\gamma$ is the ratio of specific heats, taken in this work
as $\gamma = 7/5$.

We wish to introduce an entropy pair that simultaneously symmetrizes the Euler
equations as well as the viscous term in the compressible Navier-Stokes
equations. In this case, the entropy pair is unique \cite{Hughes1986}, and is
given by
\begin{equation}
  U = -\rho s, \qquad F = - \rho \bm w s,
\end{equation}
where $s = \log \left( p \rho^{-\gamma} \right)$. Given this entropy pair, the
entropy variables can be written as
\begin{equation}
  \bm v = \left( \begin{array}{c}
    -\rho E (\gamma -1)/ p + (\gamma + 1 - s) \\
    \rho \bm w (\gamma-1) / p  \\
    -\rho (\gamma - 1) / p
  \end{array} \right).
\end{equation}
Likewise, we can transform from the entropy variables to conservative variables
by
\begin{equation}
  \bm u = \left( \begin{array}{c}
    - p v_{d+2} / (\gamma - 1) \\
     p v_j / (\gamma -1 ) \\
     p\left(1 - \frac{1}{2}\sum_{i=2}^{d+1} v_i^2 / v_{d+2}\right)/(\gamma-1)
  \end{array} \right), \qquad j=2,\ldots,d+1,
\end{equation}
where, in terms of the entropy variables, we have
\begin{equation}
  s = \gamma - v_1 + \frac{1}{2}\sum_{i=2}^{d+1} v_i^2 / v_{d+2}, \qquad
  p/(\gamma-1) = \left(
    \frac{\gamma-1}{-v_{d+2}^\gamma}
  \right)^{1/(\gamma-1)} \exp\left(\frac{-s}{\gamma-1} \right).
\end{equation}
The entropy potential flux is given by
\begin{equation}
  \psi = (\gamma - 1) \rho \bm w.
\end{equation}

\subsubsection{Entropy-conservative and entropy-stable numerical fluxes}

There has been much recent interest in the development of entropy-conservative
and entropy-stable numerical flux functions for the Euler equations
\cite{Ismail2009,Ranocha2017,Chandrashekar2013,Ray2013}. In this work, for the
volume fluxes, we will use the two-point entropy-conservative flux of
Chandrashekar \cite{Chandrashekar2013}. For $d=2$, the flux is defined as
follows. We introduce the convenient notation for the arithmetic and logarithmic
means
\begin{equation}
  \{ \phi \} = \frac{1}{2}\left( \phi_L + \phi_R \right), \qquad
  \{ \phi \}_{\log} = \frac{\phi_R - \phi_L}{\log(\phi_R) - \log(\phi_L)}.
\end{equation}
A numerically stable procedure for evaluating $\{\phi\}_{\log}$ was given by
Ismail and Roe \cite{Ismail2009}. Chandrashekar's entropy-conservative numerical
flux is given by
\begin{equation}
\begin{aligned}
  \bm f_{1,S}^{(1)} &= \{ \rho \}_{\log} \{ w_1 \} \\
  \bm f_{1,S}^{(2)} &= \{ w_1 \} \bm f_{1,S}^{(1)}+\frac{\{\rho\}}{2\{\beta\}}\\
  \bm f_{1,S}^{(3)} &= \{ w_2 \} \bm f_{1,S}^{(1)}\\
  \bm f_{1,S}^{(4)} &= \left(
    \frac{1}{2(\gamma-1)\{\beta\}_{\log}} - \frac{\{w_1^2\} + \{w_2^2\}}{2}
    \right) \bm f_{1,S}^{(1)}
    + \{ w_1 \} \bm f_{1,S}^{(2)}
    + \{ w_2 \} \bm f_{1,S}^{(3)} \\
  \bm f_{2,S}^{(1)} &= \{ \rho \}_{\log} \{ w_2 \} \\
  \bm f_{2,S}^{(2)} &= \{ w_1 \} \bm f_{2,S}^{(1)} \\
  \bm f_{2,S}^{(3)} &= \{ w_2 \} \bm f_{2,S}^{(1)}+\frac{\{\rho\}}{2\{\beta\}}\\
  \bm f_{2,S}^{(4)} &= \left(
    \frac{1}{2(\gamma-1)\{\beta\}_{\log}} - \frac{\{w_1^2\} + \{w_2^2\}}{2}
    \right) \bm f_{2,S}^{(1)}
    + \{ w_1 \} \bm f_{2,S}^{(2)}
    + \{ w_2 \} \bm f_{2,S}^{(3)}
\end{aligned}
\end{equation}
where $\beta$ is the inverse temperature, defined by
\begin{equation}
  \beta = \frac{1}{2RT} = \frac{\rho}{2p}.
\end{equation}

At element interfaces, we must introduce an entropy-stable numerical flux
function. It is shown in \cite{Chen2017} that exactly solving the Riemann
problem at element interfaces results in an entropy-stable numerical flux.
However, this process can be computationally expensive, requiring the solution
of a system of nonlinear equations for every evaluation. For this reason, we opt
to use a simple local Lax-Friedrichs (LLF) flux function, defined by
\begin{equation}
  \widehat{\bm f}(\bm u_L, \bm u_R) = \frac{1}{2}\left( \bm f(\bm u_L)
    + \bm f(\bm u_R) \right) + \frac{\lambda}{2} \left(\bm u_L - \bm u_R\right),
\end{equation}
where $\lambda$ is chosen to bound the leftmost and rightmost wave speeds in the
corresponding Riemann problem. This numerical flux function is a special case of
the Harten-Lax-Van Leer (HLL) approximate Riemann solver \cite{Harten1983},
which can be shown to be entropy-stable \cite{Chen2017}.

\subsection{Line-based DG discretization}

The line-based discontinuous Galerkin discretization (Line-DG) is constructed by
modifying the standard nodal DG discretization on tensor-product elements
(mapped quadrilaterals or hexahedra), so that a sequence of
one-dimensional Galerkin problems are solved along each coordinate dimension. To
be more specific,  the spatial domain $\Omega$ is partitioned into a conforming
mesh $\T_h = \{ K_j : 1 \leq j \leq n_e \}$, such that $\bigcup_{j=1}^{n_e} K_j
= \Omega$. Each element $K_j \in \T_h$ is taken to be the image of the reference
element $\mathcal{R} = [0,1]^d$ under a transformation map $T_j : \R^d \to
\R^d$.

We now focus on a single element, $K \in \T_h$, with corresponding
transformation map $T$. We use the convention that $\bm x = T(\bm \xi)$, and
refer to $\bm x$ as \textit{physical coordinates}, and $\bm \xi$ as
\textit{reference
coordinates}.
We wish to perform a change of variables to rewrite
the conservation law \eqref{eq:cons-law} in the reference domain $\mathcal{R}$.
Let $J$ denote the Jacobian of $T$ (referred to as the
\textit{deformation gradient}),
\begin{equation}
  J = \left( \frac{\partial x_i}{\partial \xi_j} \right),
\end{equation}
and let $g = \det(J)$. Following a standard procedure to change spatial
variables, we define
\begin{equation}
  \wt{\bm u}(\bm \xi, t) = g \bm u( T(\bm \xi), t),
\end{equation}
and
\begin{equation} \label{eq:flux-transformation}
  \wt{\bm f} (\wt{\bm u}) = g J^{-1} \bm f(\wt{\bm u}/g).
\end{equation}
Then, $\wt{\bm u}$ evolves according to the transformed hyperbolic conservation
law
\begin{equation} \label{eq:reference-eq}
  \frac{\partial \wt{\bm u}(\bm \xi, t)}{\partial t}
    + \nabla \cdot \wt{\bm f}(\wt{\bm u}(\bm \xi, t)) = 0,
\end{equation}
where, in this case, the divergence is understood to be taken with respect to
the reference coordinates $\bm \xi$. In order to introduce the Line-DG method,
we discretize equation \eqref{eq:reference-eq} directly.

For simplicity of presentation, we describe the method for $d=2$. The extension
to three or more spatial dimensions is straightforward. The reference
coordinates are written $\bm \xi = (\xi, \eta)$. We fix a polynomial degree $p
\geq 1$, and introduce nodal interpolation points $\bm \xi_{ij} = (\xi_i, \xi_j)
\in \mathcal{R}$ according to a tensor-product Gauss-Lobatto rule, for $0 \leq
i,j \leq p$. We approximate the solution using time-dependent nodal values
$\wt{\bm u}_{ij}(t) \approx \wt{\bm u}(\bm \xi_{ij},t)$. The standard Line-DG
semi-discretization in space reads:
\begin{equation}
  \frac{\partial \wt{\bm u}_{ij}}{\partial t} + \bm q_{ij} + \bm r_{ij} = 0,
\end{equation}
where $\bm q_{ij}$ and $\bm r_{ij}$ are discretizations of the derivatives
\begin{equation}
  \bm q_{ij} \approx \frac{\partial \wt{\bm f}_1}{\partial \xi}(\xi_i, \xi_j),
    \qquad
  \bm r_{ij} \approx \frac{\partial \wt{\bm f}_2}{\partial \eta}(\xi_i, \xi_j),
\end{equation}
which are both obtained through a sequence of one-dimensional Galerkin problems,
described as follows.

We begin by defining $\bm q_{ij}$, which approximates the $\xi$-derivative of
$\wt{\bm f}_1$. $\bm r_{ij}$ is defined through an analogous procedure. We fix
an index $j$, $0 \leq j \leq p$. Then, we consider all the points $\bm q_{ij}$
that lie along this line. We view these nodal values as defining an
interpolating polynomial, which we write as $\bm q_{:j}(\xi) = \sum_{i=0}^p \bm
q_{ij} \phi_i(\xi)$, where $\phi_i(\xi_k) = \delta_{ij}$ is the Lagrange
interpolating polynomial defined at the one-dimensional Gauss-Lobatto points
$\xi_k$. Similarly, we define the polynomial $\wt{\bm u}_{:j}(\xi) =
\sum_{i=0}^p \wt{\bm u}_{ij} \phi_i(\xi)$. We choose the coefficients $\bm
q_{ij}$ such that they satisfy the Galerkin problem
\begin{equation} \label{eq:q}
  \int_0^1 \bm q_{:j}(\xi) \cdot \bm\theta(\xi) \, d\xi
     = -\int_0^1 \wt{\bm f}_1(\wt{\bm u}_{:j}(\xi)) \cdot \bm\theta'(\xi) \, d\xi
     + \widehat{\wt{\bm f}_1} \cdot w |_0^1,
\end{equation}
for all test functions $\bm\theta \in \left[ \mathcal{P}^p([0,1]) \right]^{n_c}$
(the space of vector-valued polynomials of degree $p$), where $\widehat{\wt{\bm
f}_1}$ is an appropriately-defined numerical flux function. Analogously, for
fixed $i$, $\bm r_{ij}$ is defined to satisfy
\begin{equation} \label{eq:r}
  \int_0^1 \bm r_{i:}(\eta) \cdot \bm\theta(\eta) \, d\eta
     = -\int_0^1 \wt{\bm f}_2(\wt{\bm u}_{i:}(\eta)) \cdot \bm\theta'(\eta) \, d\eta
     + \widehat{\wt{\bm f}_2} \cdot w |_0^1.
\end{equation}
Our discretization is complete once we specify the numerical flux functions
$\widehat{\wt{\bm f}_1}$ and $\widehat{\wt{\bm f}_2}$. We note that given a
vector $\bm N$ normal to the reference element $\mathcal{R}$, we obtain a
transformed vector $\bm n = g J^{-T} \bm N$, normal to the physical element $K$.
We then note that
\begin{equation}
  \bm n^T \bm f
    = \left( g J^{-T} \bm N\right)^T \bm f
    = \bm N^T g J^{-1} \bm f = \bm N^T \wt{\bm f}.
\end{equation}
Thus, given a numerical flux function $\widehat{\bm f}(\bm u^-, \bm u^+)$ from
any standard discontinuous Galerkin discretization, we can define
\begin{equation}
  \bm N^T \widehat{\wt{\bm f}_{\ }}(\wt{\bm u}^-, \wt{\bm u}^+)
    =  \bm n^T \widehat{\bm f}(\bm u^-, \bm u^+),
\end{equation}
allowing us to reuse standard DG numerical flux functions for the purposes of
our Line-DG discretization.

\subsection{Discrete entropy stability}

In order to solve the one-dimensional problems \eqref{eq:q} and \eqref{eq:r}, we
discretize the integrals using a quadrature rule. For this purpose, we use a
Gauss-Lobatto rule with $\mu \geq p + 1$ points, which is exact for polynomials
of degree $2 \mu - 3$. Here we emphasize that if $\mu = p+1$, then the Line-DG
method is exactly equivalent to the Gauss-Lobatto DG spectral element method
(DG-SEM) method. However, if $\mu > p + 1$, then the method is distinct from
DG-SEM, and possesses different properties.

Using techniques similar to those developed in
\cite{Fisher2013,Carpenter2014,Chen2017,Chan2018}, we modify the Line-DG
discretization as follows in order to obtain discrete entropy stability. Let
$\xi_\alpha$, $1 \leq \alpha \leq \mu$ denote the Gauss-Lobatto quadrature
points, and let $w_\alpha$ denote the quadrature weights. We define the
rectangular $\mu \times (p+1)$ quadrature evaluation matrix $G$ by
\begin{equation} \label{eq:G-matrix}
  G_{\alpha i} = \phi_i(\xi_\alpha).
\end{equation}
Similarly, we consider the $\mu \times (p+1)$ differentiation matrix
\begin{equation} \label{eq:D-matrix}
  D_{\alpha,i} = \phi_i'(\xi_\alpha).
\end{equation}
We also define the $\mu \times \mu$ diagonal quadrature weight matrix by
\begin{equation} \label{eq:W-matrix}
  W = \mathrm{diag} ( w_1, w_2, \ldots, w_\mu ).
\end{equation}
Additionally, we define the differentiation matrix at quadrature points,
\begin{equation} \label{eq:Dt-matrix}
  \wt{D}_{\alpha,\beta} = \wt{\phi}_\beta'(\xi_\alpha),
\end{equation}
where $\wt{\phi}_\beta$ are the Lagrange interpolating polynomials
of degree $\mu - 1$ defined using the quadrature points $\xi_\alpha$. Finally,
we define the $\mu \times \mu$ boundary evaluation matrix $B$ by
\begin{equation}
  B = \mathrm{diag} ( -1, 0, \ldots, 0, 1 ).
\end{equation}

\begin{prop} \label{prop:properties}
  We briefly summarize some of the properties of the above matrices.
  \begin{enumerate}[(i)]
  \item $D = \wt{D} G$
  \item $W \wt{D} + \wt{D}^T W = B$ (summation-by-parts property)
  \item $\sum_{\beta=1}^\mu \wt{D}_{\alpha,\beta} = 0$
        (derivative of a constant is zero)
  \item $\sum_{\alpha=1}^\mu \left(W \wt{D}\right)_{\alpha,\beta} =
  \begin{cases}
    -1, &\qquad \beta=1,\\
    1, &\qquad \beta=\mu,\\
    0, &\qquad \text{otherwise.}
  \end{cases}$
  \end{enumerate}
\end{prop}
\begin{proof}
\textit{(i).}
Let $\bm a \in \R^{p+1}$. Define the polynomial $a(\xi) = \sum_{i=0}^p a_i
\phi_i(\xi)$. Then, $(D\bm a)_\alpha = a'(\xi_\alpha)$ for $1 \leq \alpha \leq
\mu$ because $D$ exactly differentiates polynomials of degree $p$. Similarly,
$\wt{D}$ exactly differentiates polynomials of degree $\mu-1 \geq p$ (given in
terms of their values at quadrature points), and thus $(\wt{D} G \bm a)_\alpha =
a'(\xi_\alpha)$.

\noindent
\textit{(ii).}
Consider two polynomials, $a(\xi) = \sum_{\alpha=1}^\mu a_\alpha
\wt{\phi}_\alpha(\xi)$ and $b(\xi) = \sum_{\alpha=1}^\mu b_\alpha
\wt{\phi}_\alpha(\xi)$. Then, $\int_0^1 a'(\xi)b(\xi)\,d\xi = \bm b^T W \wt{D}
\bm a$, since the quadrature rule is exact for polynomials of degree $2\mu -3$.
Integrating by parts, we have
\[
  \int_0^1 a'(\xi)b(\xi)\,d\xi = -\int_0^1 b'(\xi) a(\xi)\,d\xi + ab|_0^1
  = \bm b^T(-\wt{D}^T W + B)\bm a,
\]
and since $a$ and $b$ were arbitrary, we conclude $WD^T + \wt{D}^TW = B$.

\noindent
\textit{(iii).}
This is immediate since $\wt{D}$ is exact for polynomials of degree $\mu-1$.

\noindent
\textit{(iv).}
This follows from properties \textit{(ii)} and \textit{(iii)}, since
\begin{equation}
\sum_{\alpha=1}^\mu \left(W \wt{D}\right)_{\alpha,\beta} = \bm 1^T W \wt{D}
= \bm 1^T W \wt{D} + \bm 1^T \wt{D}^T W = \bm 1^T B.
\end{equation}
\end{proof}

We define the one-dimensional mass matrix by $M = G^T W G$. Then, we can write
the variational form \eqref{eq:q} as
\begin{equation} \label{eq:q-weak}
  M \bm q = -D^T W \wt{\bm f}_1 (G \bm u)
    + \widehat{\wt{\bm f}}_1,
\end{equation}
where $\bm q, \wt{\bm f}_1,$ and $\bm u$ are interpreted appropriately as
vectors of coefficients. The index $j$ has been omitted for the sake
of brevity. This is known as the \textit{weak form}. We can also define the
\textit{strong form} as follows. We rewrite \eqref{eq:q-weak} using property
\textit{(i)} above,
\begin{equation}
  M \bm q = -G^T \wt{D}^T W \wt{\bm f}_1 (G \bm u)
    + \widehat{\wt{\bm f}_1},
\end{equation}
and then perform a discrete analog of integration by parts (property
\textit{(ii)} above) to obtain
\begin{equation} \label{eq:q-strong}
  M \bm q = G^T W \wt{D} \wt{\bm f}_1 (G \bm u) - G^T B \wt{\bm f}_1(G \bm u)
    + \widehat{\wt{\bm f}_1}.
\end{equation}
Similarly, the strong form for $\bm r_{ij}$ is given by
\begin{equation} \label{eq:r-strong}
  M \bm r = G^T W \wt{D} \wt{\bm f}_2 (G \bm u) - G^T B \wt{\bm f}_2(G \bm u)
    + \widehat{\wt{\bm f}_2}.
\end{equation}

\subsubsection{Entropy and quadrature projections} \label{sec:projection}

As in the work of Chan \cite{Chan2018}, a key procedure in constructing the
entropy-stable Line-DG scheme is the \textit{entropy projection} of the
conservative variables, defined as follows. Given $\wt{\bm u}(\xi, \eta)$, we
can compute the entropy variables $\wt{\bm v} (\wt{\bm u})$. We define
$\widehat{\bm v}$ to be the discrete $L^2$ projection of $\bm v$. That is,
$\widehat{\bm v}$ is the unique bivariate polynomial of degree $p$ in each
variable satisfying
\begin{equation}
  \sum_{\alpha,\beta=1}^\mu w_\alpha w_\beta
    \widehat{\bm v}(\xi_\alpha,\xi_\beta) \cdot \bm\phi(\xi_\alpha,\xi_\beta) =
  \sum_{\alpha,\beta=1}^\mu w_\alpha w_\beta
    \wt{\bm v}(\xi_\alpha,\xi_\beta) \cdot \bm\phi(\xi_\alpha,\xi_\beta).
\end{equation}
for all test functions $\bm \phi$. We then define the \textit{entropy-projected}
conservative variables by $\widehat{\bm u} = \bm u(\widehat{\bm v})$. We note
that for \textit{e.g.} Burgers' equation with square entropy function, we have
$\bm v = \bm u$, and thus entropy projection is the identity operator.
Additionally, when $\mu = p + 1$, the discrete $L^2$ projection reduces to the
identity, and so in this case, the entropy projection is also the identity,
resulting in a simplified scheme. Evaluating the entropy-conservative and
entropy-stable fluxes at the entropy-projected values will allow us to prove
entropy stability of the discrete scheme.

An additional ingredient required for discrete entropy stability is
an operation that we introduce called a quadrature projection.
Because the Line-DG method is based on consistent integration of the
$\xi$-derivative in the $\xi$-direction, and collocation in the
$\eta$-direction, and similarly, consistent integration of the $\eta$-derivative
in the $\eta$-direction, and collocation in the $\xi$-direction, we introduce a
projection operation to allow for consistent integration of both terms in both
directions. We are interested in computing discrete integrals of the form
\begin{equation} \label{eq:2d-quadrature}
  \sum_{\alpha,\beta=1}^\mu w_\alpha w_\beta \bm q(\xi_\alpha,\xi_\beta)
    \cdot \bm\phi(\xi_\alpha,\xi_\beta)
  \qquad\text{and}\qquad
  \sum_{\alpha,\beta=1}^\mu w_\alpha w_\beta \bm r(\xi_\alpha,\xi_\beta)
    \cdot \bm\phi(\xi_\alpha,\xi_\beta),
\end{equation}
for a given bivariate polynomial $\bm\phi$. Equivalently, these integrals can be
written in the form
\begin{equation}
  \bm\phi^T (M \otimes M) \bm q, \qquad \text{and} \qquad
  \bm\phi^T (M \otimes M) \bm r,
\end{equation}
where $M$ is the one-dimensional mass matrix, $\otimes$ represents the Kronecker
product, and $\bm\phi, \bm q,$ and $\bm r$ are interpreted as vectors of the
corresponding degrees of freedom. However, the more natural line-based
quadrature takes the form
\begin{equation}
  \bm\phi^T (\wt{M} \otimes M) \bm q, \qquad \text{and} \qquad
  \bm\phi^T (M \otimes \wt{M}) \bm r,
\end{equation}
where $\wt{M}$ is the diagonal mass matrix corresponding to the nodal
interpolation points. Thus, given approximations $\bm q$ and $\bm r$ to the
$\xi$- and $\eta$-derivatives, respectively, we define their
quadrature-projected variants by
\begin{equation} \label{eq:quadrature-projection}
  \widehat{\bm q} = (\wt{M}M^{-1} \otimes I) \bm q, \qquad \text{and} \qquad
  \widehat{\bm r} = (I \otimes \wt{M}M^{-1}) \bm r,
\end{equation}
such that
\begin{equation}
  \bm\phi^T (M \otimes M) \widehat{\bm q}
   = \bm\phi^T (\wt{M} \otimes M) \bm q,
  \qquad \text{and} \qquad
  \bm\phi^T (M \otimes M) \widehat{\bm r}
   =\bm\phi^T (M \otimes \wt{M}) \bm r,
\end{equation}
allowing for computation of the discrete integrals in \eqref{eq:2d-quadrature}
using the line-based quadrature that is more natural for the Line-DG scheme.

\subsubsection{Modified scheme} \label{sec:modified}
We modify the scheme to achieve entropy stability using entropy-conservative and
entropy-stable numerical fluxes with a flux differencing approach.
This approach is closely related to the family of schemes developed
by Fisher \cite{Fisher2013}, Chan \cite{Chan2018}, Parsani \cite{Parsani2016}
and Chen \cite{Chen2017}.
Equation
\eqref{eq:q-strong} is replaced by
\begin{equation} \label{eq:modified-q}
  M \bm q = 2 G^T W \wt{D} \circ \wt{\bm f}_{1,S} (\widehat{\bm u})\bm{1}
    - G^T B \wt{\bm f}_{1,S}(\widehat{\bm u}) + \widehat{\wt{\bm f}_1},
\end{equation}
and equation \eqref{eq:r-strong} is replaced by
\begin{equation} \label{eq:modified-r}
  M \bm r = 2 G^T W \wt{D} \circ \wt{\bm f}_{2,S} (\widehat{\bm u})\bm{1}
    - G^T B \wt{\bm f}_{2,S}(\widehat{\bm u}) + \widehat{\wt{\bm f}_2},
\end{equation}
where $\circ$ denotes the Hadamard product, defined by
\begin{equation}
  ( \wt{D} \circ \wt{\bm f}_{i,S} )_{\alpha,\beta}
    = \wt{D}_{\alpha,\beta} \wt{\bm f}_{i,S}(\widehat{\bm u}(\xi_\alpha,\xi_j),
      \widehat{\bm u}(\xi_\beta,\xi_j)),
\end{equation}
where, as before, the index $j$ is omitted on the left-hand side for
the sake of conciseness. The transformed flux functions are obtained by taking
the arithmetic average of the metric terms,
\begin{equation}
  \wt{\bm f}_{i,S} (\bm u_L, \bm u_R) =
    \sum_{j=1}^d \frac{1}{2} \left(g_L J_{ij,L}^{-1} + g_R J_{ij,R}^{-1}\right)
    \bm f_{j,S} (\bm u_L, \bm u_R),
\end{equation}
where $g_L$ and $J_{ij,L}$ are defined at the points at which $\bm u_{L}$ is
evaluated, and similarly for $g_R$ and $J_{ij,R}$.
The modified entropy-stable Line-DG method reads
\begin{equation} \label{eq:entropy-stable-scheme}
  \frac{\partial \wt{\bm u}_{ij}}{\partial t} + \widehat{\bm q}_{ij}
    + \widehat{\bm r}_{ij} = 0,
\end{equation}
where $\bm q_{ij}$ is defined by \eqref{eq:modified-q} and $\bm r_{ij}$ is
defined by \eqref{eq:modified-r}, and $\widehat{\bm q}$ and $\widehat{\bm r}$
are their quadrature-projected variants, respectively. Given these definitions,
we set out to prove the accuracy, conservation, and entropy stability of the
discrete scheme.

\begin{prop}[Accuracy] \label{prop:accuracy1}
Suppose $\bm u$ is a smooth solution to \eqref{eq:cons-law}, and let $\bm u_{ij}
= \bm u(\xi_i, \xi_j)$. Define $\bm q_{ij}$ and $\bm r_{ij}$ using
\eqref{eq:modified-q} and \eqref{eq:modified-r}, respectively. Then we obtain
\begin{equation} \label{eq:linedg-no-proj}
  \frac{\partial \bm u_{ij}}{\partial t} + \bm q_{ij} + \bm r_{ij}
    = \mathcal{O}(h^{\min(p+1,\mu-1)}).
\end{equation}
\end{prop}
\begin{proof}
Since $\bm u$ is smooth, $\widehat{\wt{\bm f}_i}$ is single-valued, and by
consistency of the numerical fluxes, the two boundary terms appearing in both
\eqref{eq:modified-q} and \eqref{eq:modified-r} cancel. We begin by assuming
that $\bm v = \bm u$ and so $\widehat{\bm u} = \bm u$. Then consider $\wt{\bm
f}_{i,S}( \bm u_L, \bm u_R)$, and define $\wt{\bm f}_i(\xi) = \wt{\bm
f}_{i,S}(\bm u(\xi,\xi_j), \bm u(\xi,\xi_j))$. Then, by symmetry of $\wt{\bm
f}_{i,S}$,
\begin{equation}
  \frac{\partial \wt{\bm f}_i}{\partial \xi}(\xi) = \left.\frac{\partial \wt{\bm
  f}_{i,S}}{\partial \bm u_L} \right|_{\bm u_L,\bm u_R = \bm u(\xi,\xi_j)} +
  \left. \frac{\partial \wt{\bm f}_{i,S}}{\partial \bm u_R} \right|_{\bm u_L,\bm
  u_R = \bm u(\xi,\xi_j)} = 2 \frac{\partial \wt{\bm f}_{i,S}}{\partial \bm
  u_R} (\bm u(\xi,\xi_j), \bm u(\xi,\xi_j)).
\end{equation}
Since $\wt{D}$ is exact for polynomials of degree $\mu - 1$, we have
\[
2 \wt{D} \circ \wt{\bm f}_{i,S} (G \bm u)\bm{1}
  = 2 \frac{\partial \wt{\bm f}_{i,S}}{\partial \bm u_R}+\mathcal{O}(h^{\mu-1})
  = \frac{\partial\wt{\bm f}_i}{\partial \xi} + \mathcal{O}(h^{\mu-1}).
\]
This quantity is then projected onto the space of polynomials of degree $p$,
and we obtain
\begin{equation}
  \bm q_{:j}(\xi) = \frac{\partial \wt{\bm f}_1}{\partial \xi}(\xi, \xi_j)
    + \mathcal{O}\left(h^{\min(p+1,\mu-1)}\right),
\end{equation}
and similarly,
\begin{equation}
  \bm r_{i:}(\eta) = \frac{\partial \wt{\bm f}_2}{\partial \eta}(\xi_i, \eta)
    + \mathcal{O}\left(h^{\min(p+1,\mu-1)}\right).
\end{equation}
To complete the proof, we now note that by accuracy of the $L^2$ projection,
$\widehat{\bm v} = \bm v(\bm u) + \mathcal{O}\left(h^{p+1}\right)$
\cite{Chan2018}. Thus $\widehat{\bm u} = \bm u +
\mathcal{O}\left(h^{p+1}\right)$, and the general case follows.
\end{proof}

\begin{rem}
If $\mu = p+1$, our method is identical to the DG-SEM method, and the
truncation error is suboptimal, scaling as $\mathcal{O}(h^p)$, as shown in
\cite{Chen2017}. If $\mu > p+1$, the order of accuracy is optimal, and the
truncation error scales as $\mathcal{O}(h^{p+1})$.
\end{rem}

\begin{prop}[Accuracy of quadrature projection] \label{prop:accuracy2}
Define $\widehat{\bm q}$ and $\widehat{\bm r}$ by
\eqref{eq:quadrature-projection}. Then,
\begin{equation*}
  \frac{\partial \wt{\bm u}_{ij}}{\partial t} + \widehat{\bm q}_{ij}
    + \widehat{\bm r}_{ij}
    = \mathcal{O}(h^p).
\end{equation*}
\end{prop}
\begin{proof}
If $\mu = p+1$, then $\wt{M} = M$, and the result follows from Proposition
\ref{prop:accuracy1}. If $\mu > p +1$, we note that $\wt{M}$ is defined by
quadrature at the nodal interpolation points, which is exact for polynomials of
degree $2(p+1) - 3 = 2p - 1$. Thus, $\wt{M}$ agrees with the exact mass matrix
when applied to polynomials of degree $p-1$. So, $\widetilde{\bm q} =
M^{-1}\wt{M} \bm q = \bm q + \mathcal{O}(h^p)$, and the result follows.
\end{proof}

\begin{rem}
Proposition \ref{prop:accuracy2} implies that the quadrature projection
operation introduced in order to obtain discrete entropy stability results in a
loss of accuracy. This is verified in the numerical experiments shown in Section
\ref{sec:results}. Empirically, we observe that, for many cases, the quadrature
projection is not required for robustness of the method, and the more accurate
method defined by \eqref{eq:linedg-no-proj} may be used instead. However, for
provable discrete entropy stability, we require the method given by
\eqref{eq:entropy-stable-scheme}.
\end{rem}

\begin{prop}[Conservation] Given periodic or compactly supported boundary
conditions and an affine mesh, then $\frac{d}{dt}\int_\Omega \bm u(\bm x, t) \,d\bm x = 0$,
where $\bm u$ is obtained through the Line-DG method.
\end{prop}
\begin{proof}
We consider one element $K$, mapped to the reference element $\mathcal{R}$.
Then,
\begin{equation}
  \frac{d}{dt}\int_K \bm u(\bm x, t) \,d\bm x
  = \frac{d}{dt}\int_\mathcal{R} \wt{\bm u}(\bm x, t) \,d\bm \xi
  = - \int_\mathcal{R} \bm q_{ij} \, d\bm \xi
    - \int_\mathcal{R} \bm r_{ij} \, d\bm \xi.
\end{equation}
We discretize the two integrals on the right-hand side using appropriate
line-based quadratures. Since $\bm u, \bm q,$ and $\bm r$ are all bivariate
polynomials of degree $p$, the Gauss-Lobatto quadratures associated with both
the solution nodes and the line-based quadrature points result in exact
integration. In particular, for fixed $j$, we have
\begin{equation}
\int_0^1 \bm q_{:j}(\xi) \, d \xi \approx
  \sum_{\alpha=1}^\mu w_\alpha \bm q_{:j}(\xi_\alpha, \xi_j) =\bm1^TM\bm q_{:j}.
\end{equation}
By definition of $\bm q$ we have
\begin{equation}
  \bm 1^T M\bm q = \bm 1^T \left( 2 G^T W \wt{D} \circ \wt{\bm f}_{1,S} \bm{1}
    - G^T B \wt{\bm f}_1 - {red}\widehat{\wt{\bm f}}_{1,S} \right)
\end{equation}
The first term on the right-hand side is
\begin{align*}
  \bm 1^T \left(2 G^T W \wt{D} \circ \wt{\bm f}_{1,S} \bm{1} \right)
  &= 2 \sum_{\alpha,\beta=1}^\mu w_\alpha \wt{D}_{\alpha,\beta}
    \wt{\bm f}_{1,S}(\widehat{\bm u}(\xi_\alpha,\xi_j),
    \widehat{\bm u}(\xi_\beta, x_j)) \\
  &= \sum_{\alpha,\beta=1}^\mu w_\alpha \left(\wt{D}_{\alpha,\beta}
  + \wt{D}_{\beta,\alpha}\right)
    \wt{\bm f}_{1,S}(\widehat{\bm u}(\xi_\alpha,\xi_j),
    \widehat{\bm u}(\xi_\beta, x_j)) \\
  &= \wt{\bm f}_1(\widehat{\bm u}(\xi_\mu,\xi_j))
   - \wt{\bm f}_1(\widehat{\bm u}(\xi_1,\xi_j))
\end{align*}
by symmetry of $\wt{\bm f}_{1,S}$ and the summation-by-parts property of
$\wt{D}$. Thus, the boundary terms cancel, and we are left only with the
numerical flux term $\widehat{\wt{\bm f}_{1}}$. Summing over all elements $K$,
using that the numerical flux is single-valued, and repeating an analogous
argument for $\bm r$, we obtain the desired result.
\end{proof}

\begin{lem}\label{lem:entropy}
Given periodic or compactly supported boundary conditions and an affine mesh
$\mathcal{T}_h$, the discrete line-based approximation to
$\frac{d}{dt}\int_\Omega U(\bm x, t) \,d\bm x$ is bounded above by zero.
\end{lem}
\begin{proof}
Since $\bm v = U'(\bm u)$ we have $\frac{\partial U}{\partial t} = \bm v^T
\frac{\partial \bm u}{\partial t}$, and thus, assuming continuity in time,
\begin{equation}
  \frac{d}{dt}\int_\Omega U(\bm x, t) \,d\bm x
   = \int_\Omega \bm v(\bm x, t)^T
    \frac{\partial \bm u}{\partial t}(\bm x, t) \,d\bm x.
\end{equation}
As before, we consider a single element $K$ with corresponding transformation
mapping $T$, and rewrite the integral over the reference element, as
\begin{equation}
   \int_K \bm v(\bm x, t)^T\frac{\partial \bm u}{\partial t}(\bm x, t) \,d\bm x
   = \int_\mathcal{R} \wt{\bm v}(\bm \xi, t)^T
    \frac{\partial \wt{\bm u}}{\partial t}(\bm \xi, t) \,d\bm \xi,
\end{equation}
where $\wt{\bm u} (\bm \xi, t) = g \bm u (T(\bm \xi), t)$ and
$\wt{\bm v}(\bm \xi, t) = \bm v(T(\bm \xi), t)$. We replace
$\frac{\partial\wt{\bm u}}{\partial t}$ by $-(\bm q+\bm r)$, and discretize the
above integrals using the appropriate line-based quadratures.

To begin, we consider, for fixed index $j$,
\begin{equation}
\int_0^1 \wt{\bm v}(\xi,\xi_j)^T \bm q_{:j}(\xi) \, d\xi \approx
  \sum_{\alpha=1}^\mu w_\alpha \wt{\bm v}(\xi_\alpha, \xi_j)^T
    \bm q_{:j}(\xi_\alpha, \xi_j).
\end{equation}
Let $\widehat{\bm v}$ denote the $L^2$ projection of $\bm v(\xi, \xi_j)$ onto
$[\mathcal{P}^p([0,1])]^{n_c}$. Then, since $q_{:,j}$ is itself a polynomial of
degree $p$, we have
\begin{equation}
\sum_{\alpha=1}^\mu w_\alpha \wt{\bm v}(\xi_\alpha, \xi_j)^T
    \bm q_{:j}(\xi_\alpha, \xi_j)
= \sum_{\alpha=1}^\mu w_\alpha \widehat{\bm v}(\xi_\alpha, \xi_j)^T
    \bm q_{:j}(\xi_\alpha, \xi_j)
= \widehat{\bm v}^T M \bm q.
\end{equation}
Using the definition of $\bm q$ in \eqref{eq:modified-q}, we have
\begin{equation} \label{eq:entropy-integral}
  \widehat{\bm v}^T M \bm q =
    \widehat{\bm v}^T\left(2 G^T W \wt{D} \circ \wt{\bm f}_{1,S}
      (\widehat{\bm u})\bm{1}
      - G^T B \wt{\bm f}_{1,S}(\widehat{\bm u})+\widehat{\wt{\bm f}_1} \right).
\end{equation}
We use $2W\wt{D} = W\wt{D} - \wt{D}^TW + B$ to rewrite the first term on the
right-hand side as
\begin{equation}
  \sum_{\alpha,\beta=1}^\mu \widehat{\bm v}(\xi_\alpha,\xi_j)^T
    (w_\alpha \wt{D}_{\alpha,\beta} - w_\beta \wt{D}_{\beta,\alpha}
    + B_{\alpha\beta})
    \wt{\bm f}_{1,S}(\widehat{\bm u}(\xi_\alpha,\xi_j),
                     \widehat{\bm u}(\xi_\beta, x_j)).
\end{equation}
The boundary term exactly cancels the second term on the right-hand side of
\eqref{eq:entropy-integral}. We reindex and use symmetry of $\wt{\bm f}_{1, S}$
to write the remaining terms as
\begin{equation}
  \sum_{\alpha,\beta=1}^\mu \left(
     \widehat{\bm v}(\xi_\alpha,\xi_j)
   - \widehat{\bm v}(\xi_\beta,\xi_j) \right)^T w_\alpha \wt{D}_{\alpha,\beta}
    \wt{\bm f}_{1,S}(\widehat{\bm u}(\xi_\alpha,\xi_j),
                     \widehat{\bm u}(\xi_\beta, x_j)).
\end{equation}
We assume that the mesh is affine, and so $gJ^{-1}$ is constant. Thus, $\wt{\bm
f}_{1,S} = gJ^{-1}_{11} \bm f_{1,S} + gJ^{-1}_{12} \bm f_{2,S}$. Since
$\widehat{\bm u} = \bm u(\widehat{\bm v})$, we use the entropy conservation of
the two-point flux to write this sum as
\begin{align*}
  &\sum_{\alpha,\beta=1}^\mu  w_\alpha \wt{D}_{\alpha,\beta} \left(
  gJ^{-1}_{11} \left( \psi_{1,\alpha} - \psi_{1,\beta} \right)
  + gJ^{-1}_{21} \left( \psi_{2,\alpha} - \psi_{2,\beta} \right) \right) \\
 &\qquad = - \sum_{\alpha,\beta=1}^\mu  w_\alpha \wt{D}_{\alpha,\beta} \left(
  gJ^{-1}_{11} \psi_{1,\beta} + gJ^{-1}_{21} \psi_{2,\beta} \right) \\
 &\qquad = gJ^{-1}_{11} \left( \psi_{1,1} - \psi_{1,\mu} \right)
  + gJ^{-1}_{21} \left( \psi_{2,1} - \psi_{2,\mu} \right),
\end{align*}
where we used properties \textit{(iii)} and \textit{(iv)} of Proposition
\ref{prop:properties}. Therefore, the total entropy production corresponding
to $\bm q_{:j}$ for the element $K$ is given by
\begin{equation}
  gJ^{-1}_{11} \left( \widehat{\bm v}^T_\mu \widehat{\bm f}_{1,\mu}
   - \widehat{\bm v}^T_1 \widehat{\bm f}_{1,1}
   + \psi_{1,1} - \psi_{1,\mu} \right)
    + gJ^{-1}_{21} \left( \widehat{\bm v}^T_\mu \widehat{\bm f}_{2,\mu}
    - \widehat{\bm v}^T_1 \widehat{\bm f}_{2,1}
    + \psi_{2,1} - \psi_{2,\mu} \right).
\end{equation}
We now sum the contributions along a shared edge of two elements, $K_L$ and
$K_R$. We obtain
\begin{equation}
  gJ^{-1}_{11} \left( \widehat{\bm v}^T_L \widehat{\bm f}_1
   - \widehat{\bm v}^T_R \widehat{\bm f}_1
   + \psi_{1,R} - \psi_{1,L} \right)
    + gJ^{-1}_{21} \left( \widehat{\bm v}^T_L \widehat{\bm f}_2
    - \widehat{\bm v}^T_R \widehat{\bm f}_2
    + \psi_{2,R} - \psi_{2,L} \right) \leq 0
\end{equation}
using the entropy stability of the numerical flux function. We repeat a similar
argument for the term $\bm r_{ij}$.
\end{proof}
\begin{rem}
We note that the assumption that the mesh is affine in the above proposition can
be relaxed to allow for curved geometries, under the condition that the metric
terms satisfy a discrete version of the geometric conservation law
\cite{Kopriva2006}.
\end{rem}

\begin{prop}[Discrete entropy stability] \label{prop:entropy-stability}
Given periodic or compactly supported boundary conditions and an affine mesh
$\mathcal{T}_h$, we have
\begin{equation}
  \frac{d}{dt} \sum_{K\in\mathcal{T}_h} \sum_{\alpha,\beta=1}^\mu
    w_\alpha w_\beta U(\wt{\bm u}_K(\xi_\alpha,\xi_\beta)) \leq 0,
\end{equation}
where $\wt{\bm u}$ is given by \eqref{eq:entropy-stable-scheme}.
\end{prop}
\begin{proof}
Noting that $\frac{\partial U}{\partial t} = \bm v^T \frac{\partial \bm
u}{\partial t}$, we have
\begin{align*}
  \frac{d}{dt} \sum_{K\in\mathcal{T}_h}
  \sum_{\alpha,\beta=1}^\mu w_\alpha w_\beta
    U(\wt{\bm u}_K(\xi_\alpha,\xi_\beta))
  &= \sum_{K\in\mathcal{T}_h} \sum_{\alpha,\beta=1}^\mu w_\alpha w_\beta
  \bm v(\wt{\bm u}_K(\xi_\alpha,\xi_\beta))^T\left(
  \widehat{\bm q}_K(\xi_\alpha,\xi_\beta)
  + \widehat{\bm r}_K(\xi_\alpha,\xi_\beta) \right) \\
  &= \sum_{K\in\mathcal{T}_h}
    \widehat{\bm v}_K^T \left(M \otimes M\right) \left( \widehat{\bm q}_K
    + \widehat{\bm r}_K \right) \\
  &= \sum_{K\in\mathcal{T}_h}
    \widehat{\bm v}_K^T \left(\wt{M} \otimes M\right) \widehat{\bm q}_K
    + \widehat{\bm v}_K^T \left(M \otimes \wt{M}\right) \widehat{\bm r}_K \\
  & \leq 0,
\end{align*}
using the definition of the quadrature projection and Lemma \ref{lem:entropy}.
\end{proof}

\section{Implementation and computational cost}
\label{sec:implementation}

The implementation of the Line-DG method is relatively simple, and
benefits greatly from the reuse of much of the infrastructure required for a
standard DG method: the fluxes, numerical flux functions, boundary conditions,
and metric terms remain, for the most part, unchanged. In fact, some features of
the method allow for significant simplifications: all volume integrals are
replaced with one-dimensional integrals, and no surface or face integrals are
required. A key feature of the Line-DG method compared with traditional DG
methods is its reduced computational cost. This reduced cost is attributed both
to the smaller number of flux evaluations and less expensive interpolation and
integration operations, when compared with standard DG.

We first compare the total number of flux evaluations required by each method.
The entropy-stable DG method requires the two-point flux function $\bm f_S$ to
be evaluated at all pairs of quadrature points. In a $d$-dimensional
tensor-product element with a quadrature rule based on one-dimensional Gaussian
quadrature with $\mu \geq p+1$ points, there are $\mu^d$ such points, requiring
$\mu^{2d}$ evaluations of $\bm f_S$. Symmetry of the flux $\bm f_S$ can be used
to reduce this number by about a factor of two. On the other hand, the Line-DG
method requires the evaluation of $\bm f_S$ at all pairs of quadrature points
along each line of nodes with an element. There are $d (p+1)^{d-1}$ such lines,
necessitating $d (p+1)^{d-1} \mu^2$ flux evaluations. As in the DG case, this
can be reduced by about a factor of two by exploiting the symmetry of $\bm f_S$.
The DG-SEM method can likewise take advantage of the same directional splitting,
requiring $d (p+1)^{d+1}$ total flux evaluations. Similar scaling is also required
for the related method of Parsani et al.\ \cite{Parsani2016}.
In Table \ref{tab:flux-evals} we compare the number of flux evaluations for
quadrature rules with $\mu = p + 2$ points and $\mu =
\left\lceil\frac{3}{2}(p+1)\right\rceil$ points (e.g.\ to correctly integrate a
quadratic nonlinearity in the flux function) for the specific case of $d=3$.

\begin{table}
\caption{Number of evaluations of $\bm f_S$ required by the full DG method with
general quadrature rules and the Line-DG method in 3 spatial dimensions, for
one-dimensional quadrature rules with $\mu$ points.}
\label{tab:flux-evals} \small \centering
\begin{tabular}{llllll}
\toprule
& & $p=3$ & $p=4$ & $p=5$ & $p=6$ \\
\midrule
\multirow{2}{*}{$\mu=p+2$}
& Full DG & 15{,}625 & 46{,}656 & 117{,}649 & 262{,}144 \\
& Line-DG & 1{,}200 & 2{,}700 & 5{,}292 & 9{,}408 \\
\midrule
\multirow{2}{*}{$\mu=\left\lceil\frac{3}{2}(p+1)\right\rceil$}
& Full DG & 46{,}656 & 262{,}144 & 531{,}441 & 1{,}771{,}561 \\
& Line-DG & 1{,}728 & 4{,}800 & 8{,}748 & 17{,}787 \\
\midrule
& DG-SEM & 768 & 1{,}875 & 3{,}888 & 7{,}203 \\
\bottomrule
\end{tabular}
\end{table}

Additionally, the Line-DG and DG-SEM methods require the evaluation of the
numerical flux function $\widehat{\bm f}$ only at nodal points lying on each face of a
given element, resulting in $(p+1)^{d-1}$ evaluations of $\widehat{\bm f}$. On the other
hand, a standard DG method requires the integration of $\widehat{\bm f}$ according to a
$(d-1)$-dimensional quadrature rule, resulting in $\mu^{d-1}$ evaluations, where $\mu \geq
p+1$.

We now consider the interpolation and integration operations. Since the Line-DG
method is based on the evaluation of one-dimensional integrals, the
interpolation and integration operations are equivalent to those of a standard
one-dimensional DG method. The interpolation operator $G$ is defined by
\eqref{eq:G-matrix}, and is a one-dimensional Vandermonde matrix. The
integration operator is given by $G^T W$, where $W$ is a diagonal matrix
consisting of the quadrature weights. The entropy-stable differentiation
operator is defined by $\widetilde{D}$, as given by \eqref{eq:Dt-matrix}. The
operators are identical among all elements, and along each of the spatial
dimensions. The complexity of applying these operators to an entire element is
linear in $p$ per degree of freedom, which is the same as a sum-factorized DG
method, although the implementation is significantly simpler
\cite{Pazner2018,Vos2010,Orszag1980}. Additionally, because these operators
are identical along each spatial dimension, the implementation can benefit from
batched BLAS-3 operations, by considering the degrees of freedom along all
lines of nodes within an element as a matrix of size $(p+1)\times d(p+1)^{d-1}$.
This is particularly important on modern computer architectures, for which
matrix-vector products are often memory-bound \cite{Srensen2012}. In
contrast, a collocated DG-SEM gives $G = I$, allowing one to avoid
the computation of interpolation operators.

Finally, we consider the storage cost of the method. The number of degrees of
freedom is the same as a standard DG method. However, the cost of storing
precomputed metric terms and interpolation matrices is reduced. All stored
matrices in the Line-DG method are either $(p+1)\times(p+1)$ or $\mu \times \mu$
in size, as compared with $(p+1)^d \times (p+1)^d$ or $\mu^d \times \mu^d$ for a
full DG method. Additionally, the metric terms need to be stored at quadrature
points. For the DG method, the inverse of the transformation Jacobian matrix is
a $d \times d$ matrix, thus requiring the storage of $d^2 \mu^d$ terms. In the
case of Line-DG, along each line of nodes, only one row of the inverse Jacobian
corresponding to the given coordinate dimension is required.
For example, in two dimensions, the method requires the computation of
$\partial \wt{\bm{f}}_1 / \partial \xi$ along horizontal lines and $\partial \wt{\bm{f}}_2
/ \partial \eta$ along vertical lines. According to the transformation
\eqref{eq:flux-transformation}, $\wt{\bm{f}}_1$ and $\wt{\bm{f}}_2$ are given by
\begin{equation}
   \wt{\bm{f}}_1 = g \sum_{j=1}^2 J^{-1}_{1j} \bm f_j,
   \qquad
   \text{and}
   \qquad
   \wt{\bm{f}}_2 = g \sum_{j=1}^2 J^{-1}_{2j} \bm f_j,
\end{equation}
and so only the metric terms $gJ^{-1}_{1j}$ must be stored along horizontal lines, and
likewise only the terms $gJ^{-1}_{2j}$ must be stored along vertical lines.
Each line consists of $\mu$ quadrature points, and there are $d (p+1)^{d-1}$ such lines,
thus necessitating the storage of $d^2 (p+1)^{d-1} \mu$ terms. Since $\mu \geq p + 1$, we
see that the Line-DG method enjoys reduced storage costs when compared with the standard
DG method.

\section{Numerical Results} \label{sec:results}

In the following sections we study the line-based DG scheme analysis above, applied to
several test problems in one, two and three spatial dimensions. In all of the below
examples, we integrate in time using an explicit Runge-Kutta scheme. We note that we have
only semi-discrete entropy-stability in the above analysis, and not fully-discrete
entropy-stability. However, we choose a time step sufficiently small such that the
temporal errors are all negligible compared with the spatial errors.

\subsection{1D sinusoidal Burgers'}

We begin with a simple test case for Burgers' equation
\begin{equation}
  u_t + \left(\tfrac{1}{2} u^2 \right)_x = 0,
\end{equation}
with periodic boundary conditions and sinusoidal initial conditions,
$u_0(x) = \frac{1}{2} + \sin(x)$. We choose the entropy function $U=\half u^2$.
The corresponding entropy-conservative numerical flux is defined by
(Cf.\ \cite{Tadmor2003})
\begin{equation}
  F_S(u_L, u_R) = \frac{1}{6}\left(u_L^2 + u_L u_R + u_R^2 \right),
\end{equation}
which corresponds to the standard skew-symmetric split splitting of Burgers'
equation \cite{Chen2017,Gassner2013}.
At element interfaces, the entropy-stable numerical flux function is defined by
solving the Riemann problem exactly. At $t=0.5$ the solution is still smooth,
however by $t=1.5$ the solution has developed a discontinuity. We begin by
computing the $L^\infty$ error of the Line-DG method at $t=0.5$, where we
compare to the exact solution, which is computed by following the characteristic
curves backwards in time.

We compare the accuracy of the Line-DG method using $\mu = p+1$ (\textit{i.e.}
DG-SEM) with the Line-DG method using $\mu > p + 1$. Only negligible differences
were observed for different values of $\mu$ greater than $p+1$. We additionally
remark that in the 1D case, the Line-DG method with is identical to a standard
DG method with the specified quadrature rule. For DG-SEM, our results agree with
those of Chen and Shu \cite{Chen2017}. In this case, we observe suboptimal
$\mathcal{O}(h^p)$ convergence in the $L^\infty$ error. For $\mu > p+1$, we
recover optimal $\mathcal{O}(h^{p+1})$ convergence in the $L^\infty$ norm.

Additionally, we compare the solution quality at $t=1.5$, after the shock has
developed. We compare the $\mu = p+1$ DG-SEM method to the Line-DG $\mu > p+1$
method, for $p=5$ and number of elements $N=120$. Both solutions display
non-physical oscillations in the vicinity of the shock, although the
oscillations are slightly smaller in magnitude when using $\mu > p+1$. In both
cases, the cell averages of the approximate solution well-approximate the true
solution.

\begin{figure}
\includegraphics{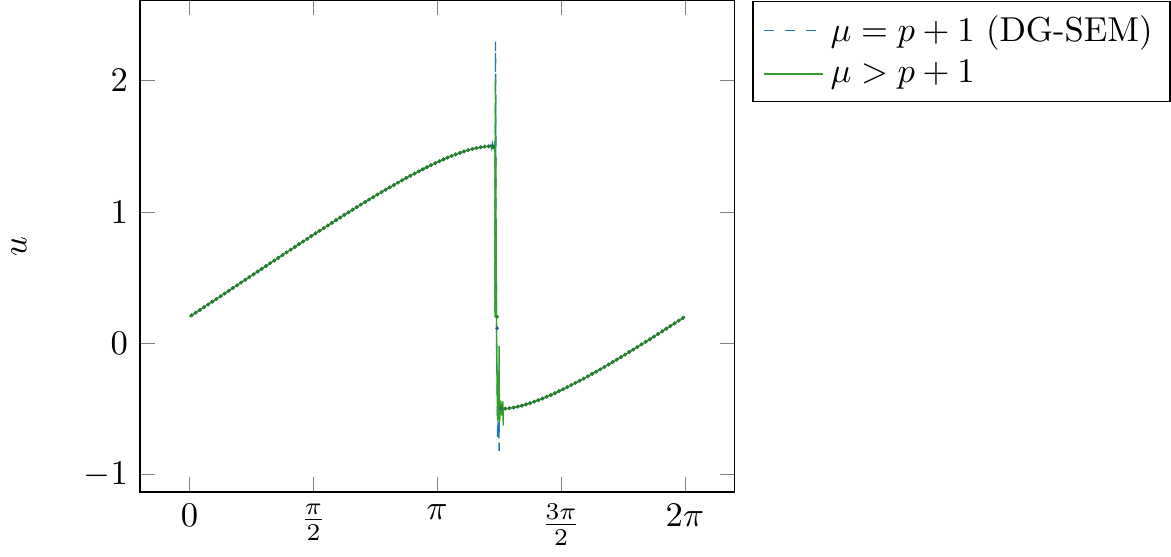}
\caption{Comparison of $\mu=p+1$ (DG-SEM) with $\mu>p+1$ for Burgers' equation,
         at $t=1.5$. Cell averages are indicated with plot markers.}
\label{fig:burgers-shock}
\end{figure}

\begin{table}
\caption{$L^\infty$ error and convergence rates for the smooth solution to
         1D Burgers' equation at $t=0.5$.}
\begin{tabular}{ll@{\hskip 16pt}ll@{\hskip 16pt}ll}
\toprule
& & \multicolumn{2}{c}{$\mu=p+1$} & \multicolumn{2}{c}{$\mu>p+1$} \\
& $N$ & Error & Rate & Error & Rate \\
\midrule
\multirow{4}{*}{$p=2$}
& 40  & $3.27\times10^{-3}$ & ---  & $7.59\times10^{-4}$ & ---  \\
& 80  & $7.92\times10^{-4}$ & 2.04 & $1.06\times10^{-4}$ & 2.84 \\
& 160 & $2.08\times10^{-4}$ & 1.93 & $1.47\times10^{-5}$ & 2.86 \\
& 320 & $5.10\times10^{-5}$ & 2.03 & $1.93\times10^{-6}$ & 2.93 \\
\midrule
\multirow{4}{*}{$p=3$}
& 40  & $1.65\times10^{-4}$ & ---  & $3.84\times10^{-5}$ & ---  \\
& 80  & $1.62\times10^{-5}$ & 3.35 & $2.54\times10^{-6}$ & 3.92 \\
& 160 & $1.31\times10^{-6}$ & 3.63 & $1.80\times10^{-7}$ & 3.82 \\
& 320 & $9.34\times10^{-8}$ & 3.81 & $1.17\times10^{-8}$ & 3.94 \\
\midrule
\multirow{4}{*}{$p=4$}
& 40  & $1.13\times10^{-5}$ & ---  & $3.09\times10^{-6}$  & ---  \\
& 80  & $7.16\times10^{-7}$ & 3.98 & $1.08\times10^{-7}$  & 4.85 \\
& 160 & $4.34\times10^{-8}$ & 4.04 & $3.90\times10^{-9}$  & 4.78 \\
& 320 & $2.62\times10^{-9}$ & 4.05 & $1.29\times10^{-10}$ & 4.92 \\
\midrule
\multirow{4}{*}{$p=5$}
& 40  & $7.12\times10^{-7}$  & ---  & $1.64\times10^{-7}$  & ---  \\
& 80  & $1.87\times10^{-8}$  & 5.25 & $3.49\times10^{-9}$  & 5.55 \\
& 160 & $3.93\times10^{-10}$ & 5.57 & $6.07\times10^{-11}$ & 5.85 \\
& 320 & $1.32\times10^{-11}$ & 4.89 & $1.03\times10^{-12}$ & 5.87 \\
\bottomrule
\end{tabular}
\end{table}

\subsection{1D shock tube}

In this section, we consider both the classic shock tube problem of Sod, as well
as a slightly modified Mach 2 shock tube problem. Both problems are solved on
the domain $\Omega = [-0.5, 0.5]$, and the initial conditions for both problems
posses a discontinuity at the origin,
\begin{equation}
  \bm u_0(x) = \begin{cases}
    \bm u_L, \qquad x < 0, \\
    \bm u_R, \qquad x \geq 0.
  \end{cases}
\end{equation}
The initial conditions for Sod's shock tube are given by
\begin{equation}
  \bm u_L =
  \left( \begin{array}{c} \rho_L \\ w_L \\ p_L \end{array} \right)
  = \left( \begin{array}{c} 1 \\ 0 \\ 1 \end{array} \right), \qquad
  \bm u_R =
  \left( \begin{array}{c} \rho_R \\ w_R \\ p_R \end{array} \right)
  = \left( \begin{array}{c} 1/8 \\ 0 \\ 1/10 \end{array} \right).
\end{equation}
The initial conditions for the Mach 2 shock are given by
\begin{equation}
  \bm u_L =
  \left( \begin{array}{c} \rho_L \\ w_L \\ p_L \end{array} \right)
  = \left( \begin{array}{c} 1.162 \\ 0 \\ 4.5 \end{array} \right), \qquad
  \bm u_R =
  \left( \begin{array}{c} \rho_R \\ w_R \\ p_R \end{array} \right)
  = \left( \begin{array}{c} 1/8 \\ 0 \\ 1/10 \end{array} \right).
\end{equation}
Both problems give rise to a rarefaction wave, a contact discontinuity, and a
shock. We solve both problems using $p=6$ polynomials with 48 elements, and
integrate in time until $t=0.1$. A comparison of the numerical solutions with
the exact solutions is shown in Figure \ref{fig:shock-tubes}. For the Sod shock
tube problem, the DG-SEM method and Line-DG method with $\mu > p+1$ give
comparable results. However, for the Mach 2 shock problem, the DG-SEM method
gives rise to significantly more prominent oscillations, demonstrating a
potential advantage of integrating with higher-accuracy quadrature rules. This
phenomenon is particularly noticeable in the velocity component of the solution.

\begin{figure}
\includegraphics{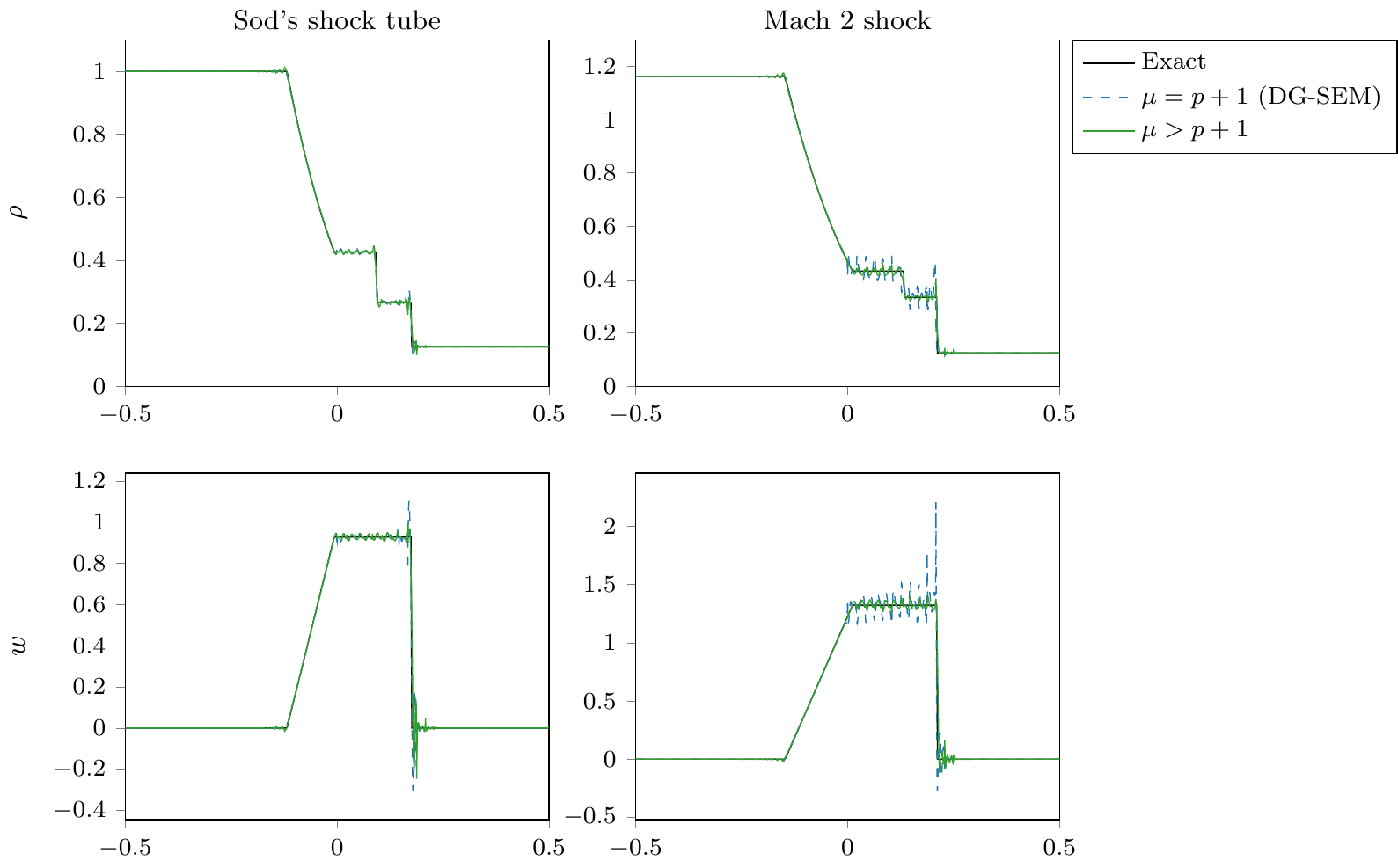}
\caption{Numerical and exact solutions to the shock tube problems at $t=0.1$.
         Density is shown in the top row and velocity is shown in the bottom
         row.}
\label{fig:shock-tubes}
\end{figure}

\subsection{2D Burgers' equation}
For the first two-dimensional test problem, we consider the Line-DG method
applied to the 2D Burgers' equation,
\begin{equation}
  u_t + \nabla\cdot\left( \tfrac{1}{2} \bm\beta u^2 \right) = 0,
\end{equation}
with constant velocity vector $\bm\beta = (1,1)$. We choose the smooth initial
conditions $u_0(x,y) = \half + \sin(x)\cos(y)$, and integrate in time until
$t=0.5$. At this point, the solution remains smooth, and as in the 1D case, we
can obtain the exact solution by tracing backwards along characteristic lines.

We compare the $L^2$ error obtained using the Line-DG method and the DG-SEM
($\mu=p+1$) method. We also consider the Line-DG method without the final
quadrature projection. Results are shown in Table \ref{tab:burgers2derrors}. For
both the DG-SEM and Line-DG methods, we observe slightly sub-optimal
convergence: approximately $\mathcal{O}(h^{p+1/2})$. Without the quadrature
projection operation, the convergence is approximately $\mathcal{O}(h^{p+1})$,
in accordance with Propositions \ref{prop:accuracy1} and \ref{prop:accuracy2}.
For each test case, the error obtained using the Line-DG method is smaller by
approximately a factor of two, and the rate of convergence appears
to be slightly faster than that of DG-SEM for a majority of cases. We also
investigate the entropy dissipation of each of these methods. Since both methods
are entropy-stable, the total entropy must be monotonically non-increasing. For
this test case the solution is smooth, and so the entropy for the exact solution
remains constant. In Figure \ref{fig:burgers2dentropy}, we compare the relative
deviation from the initial entropy, measured by $\int_\Omega \left( U(\bm
x,t)-U(\bm x, 0) \right)\,d\bm x/\int_\Omega U(\bm x, 0)\,d\bm x$, where
$U(x,t)$ is the square entropy, $U(x,t) = \half u(x,t)^2$. We numerically
observe the entropy stability of both methods, but note that the Line-DG method
dissipates less entropy than the DG-SEM method.

\begin{table}
\caption{$L^2$ error and convergence rates for the smooth solution to
         2D Burgers' equation at $t=0.5$.}
\label{tab:burgers2derrors}
\begin{tabular}{ll@{\hskip 16pt}ll@{\hskip 16pt}ll@{\hskip 16pt}ll}
\toprule
& & \multicolumn{2}{c}{$\mu=p+1$} & \multicolumn{2}{c}{$\mu>p+1$} & \multicolumn{2}{c}{No projection} \\
& $N$ & Error & Rate & Error & Rate & Error & Rate \\
\midrule
\multirow{4}{*}{$p=2$}
& 12 & $4.03\times10^{-2}$ & ---  & $2.65\times10^{-2}$ & ---  & $2.20\times10^{-3}$ & ---  \\
& 24 & $9.20\times10^{-3}$ & 2.13 & $4.11\times10^{-3}$ & 2.69 & $2.20\times10^{-3}$ & 2.90 \\
& 48 & $1.83\times10^{-3}$ & 2.33 & $6.72\times10^{-4}$ & 2.61 & $3.06\times10^{-4}$ & 2.84 \\
& 96 & $3.43\times10^{-4}$ & 2.41 & $1.11\times10^{-4}$ & 2.60 & $4.21\times10^{-5}$ & 2.86 \\
\midrule
\multirow{4}{*}{$p=3$}
& 12 & $8.03\times10^{-3}$ & ---  & $3.54\times10^{-3}$ & ---  & $2.46\times10^{-3}$ & ---  \\
& 24 & $8.01\times10^{-4}$ & 3.32 & $3.73\times10^{-4}$ & 3.25 & $2.08\times10^{-4}$ & 3.57 \\
& 48 & $7.66\times10^{-5}$ & 3.39 & $3.47\times10^{-5}$ & 3.43 & $1.54\times10^{-5}$ & 3.76 \\
& 96 & $6.67\times10^{-6}$ & 3.52 & $3.14\times10^{-6}$ & 3.47 & $1.08\times10^{-6}$ & 3.83 \\
\midrule
\multirow{4}{*}{$p=4$}
& 12 & $1.37\times10^{-3}$ & ---  & $8.02\times10^{-4}$ & ---  & $6.40\times10^{-4}$ & ---  \\
& 24 & $9.46\times10^{-5}$ & 3.85 & $3.67\times10^{-5}$ & 4.45 & $2.51\times10^{-5}$ & 4.67 \\
& 48 & $4.84\times10^{-6}$ & 4.29 & $1.60\times10^{-6}$ & 4.52 & $9.60\times10^{-7}$ & 4.71 \\
& 96 & $2.27\times10^{-7}$ & 4.42 & $7.05\times10^{-8}$ & 4.50 & $3.43\times10^{-8}$ & 4.81 \\
\midrule
\multirow{4}{*}{$p=5$}
& 12 & $3.78\times10^{-4}$ & ---  & $1.86\times10^{-4}$ & ---  & $1.54\times10^{-4}$ & ---  \\
& 24 & $1.19\times10^{-5}$ & 4.99 & $4.90\times10^{-6}$ & 5.25 & $3.52\times10^{-6}$ & 5.45 \\
& 48 & $3.15\times10^{-7}$ & 5.24 & $1.08\times10^{-7}$ & 5.50 & $6.77\times10^{-8}$ & 5.70 \\
& 96 & $7.06\times10^{-9}$ & 5.48 & $2.35\times10^{-9}$	& 5.53 & $1.23\times10^{-9}$ & 5.79 \\
\bottomrule
\end{tabular}
\end{table}

\begin{figure}
\includegraphics{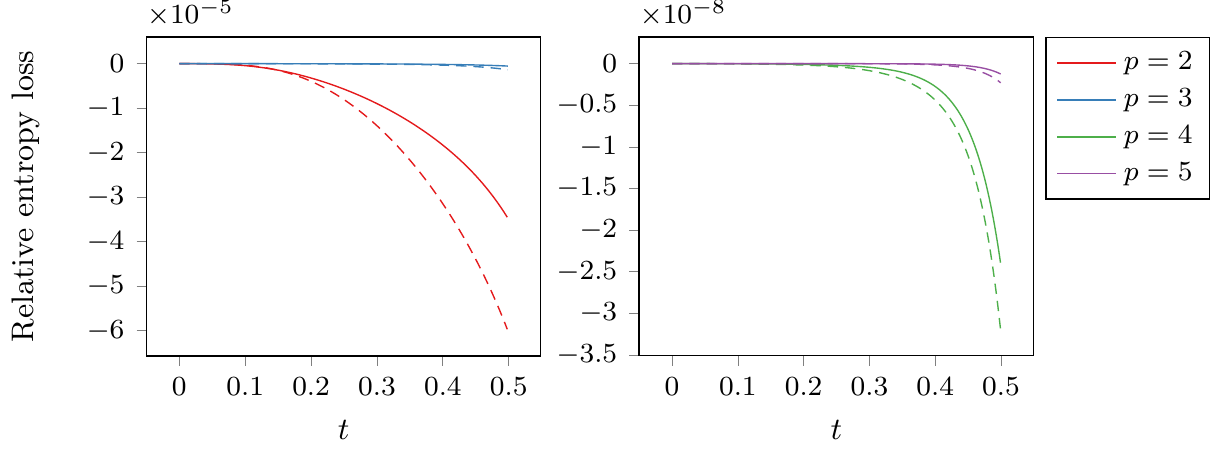}
\caption{Relative entropy loss for the 2D Burgers' equation, with $U=\half u^2$,
         for $N=12$ and $p=2,3,4,5$. Line-DG shown with solid lines, DG-SEM shown with
         dashed lines.}
\label{fig:burgers2dentropy}
\end{figure}

As a final investigation into the entropy conservation and stability properties
of the method, we consider using the entropy conservative numerical flux function
at element interfaces by defining, as in \cite{Gassner2013},
\begin{equation} \label{eq:f-econ}
   \widehat{\bm f_S}(u_L, u_R) \cdot \bm n = \frac{1}{6} \left(
      u_L^2 + u_L u_R + u_R^2 \right) \bm\beta \cdot \bm n.
\end{equation}
This is not a physically relevant choice of numerical flux function for Burgers'
equation because it fails to dissipate entropy across discontinuities. However,
it is useful to test the numerical properties of the scheme. We now solve the
same problem as above, but integrate in time until $t = 1.25$, at which point
the solution has developed discontinuities, and thus the entropy of the true
solution is less than the entropy of the initial condition. In Figure
\ref{fig:burgers2dentropycons} we compare the relative entropy loss of the
Line-DG method using both (entropy stable) exact Riemann solver and the entropy
conservative numerical flux function \eqref{eq:f-econ} at element interfaces.
These results confirm that using the entropy conservative numerical flux
function gives unchanged entropy (to machine precision) throughout the duration
of the simulation, numerically confirming the results shown in Proposition
\ref{prop:entropy-stability}.

\begin{figure}
\centering
\includegraphics{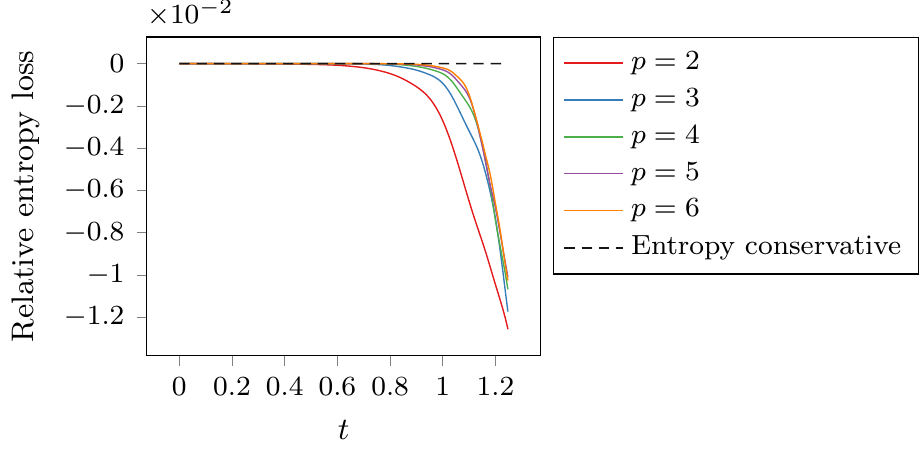}
\caption{Relative entropy loss for the 2D Burgers' equation at $t=1.25$,
         for $N=12$ and $p=2,3,4,5,6$ using the Line-DG method. Results
         obtained using the exact Riemann solver are shown in solid lines,
         results using the entropy conservative flux are shown in dashed
         lines.}
\label{fig:burgers2dentropycons}
\end{figure}

\subsection{2D isentropic vortex}

For this test problem, we study the accuracy of the Line-DG
method applied to the isentropic Euler vortex test case \cite{Wang2013}. This
problem consists of an isentropic vortex that is advected with the freestream
velocity, and is often used as a smooth benchmark problem
\cite{Pazner2017,Zahr2013}. The spatial domain is taken to be
$[0,20]\times[0,15]$. The vortex is initially centered at $(x_0, y_0) = (5,5)$,
and is advected at an angle of $\theta$. The exact solution at position and time
$(x,y,t)$ is given by
\begin{equation}
\begin{aligned}
    \rho(x,y,t) &= \rho_\infty \left( 1 -
        \frac{\epsilon^2 (\gamma - 1)M^2_\infty}{8\pi^2} \exp(f(x,y,t))
        \right)^{\frac{1}{\gamma-1}}, \\
    w_1(x,y,t) &= w_\infty \left( \cos(\theta) -
        \frac{\epsilon ((y-y_0) - \overline{w_2} t)}{2\pi r_c}
        \exp\left( \frac{f(x,y,t)}{2} \right) \right),\\
    w_2(x,y,t) &= w_\infty \left( \sin(\theta) - \frac{\epsilon ((x-x_0) -
        \overline{w_1} t)}{2\pi r_c}
        \exp\left( \frac{f(x,y,t)}{2} \right) \right),\\
    p(x,y,t) &= p_\infty \left( 1 -
        \frac{\epsilon^2 (\gamma - 1)M^2_\infty}{8\pi^2} \exp(f(x,y,t))
        \right)^{\frac{\gamma}{\gamma-1}},
\end{aligned}
\end{equation}
where $f(x,y,t) = (1 - ((x-x_0) - \overline{w_1}t)^2 - ((y-y_0) -
\overline{w_2}t)^2)/r_c^2$, $M_\infty$ is the freestream Mach number, and
$w_\infty, \rho_\infty,$ and $p_\infty$ are the freestream velocity magnitude,
density, and pressure, respectively. The freestream velocity is given by
$(\overline{w_1}, \overline{w_2}) = w_\infty (\cos(\theta), \sin(\theta))$. The
strength of the vortex is given by $\epsilon$, and its size by $r_c$. We choose
the parameters to be $M_\infty = 0.5$, $w_\infty = 1$, $\theta = \arctan(1/2)$,
$\epsilon = 5/(2\pi)$, and $r_c = 1.5$. We integrate the equations until $t =
5$.

We perform a convergence study to investigate the effects of the projection
operation described in Section \ref{sec:projection} on the accuracy of the
method. Additionally, for comparison we consider a fully-integrated standard DG
method. The local Lax-Friedrichs numerical flux function was used for all
methods. The $L^\infty$ error at the final time is shown in Figure
\ref{fig:everrors}. We observe that the Line-DG method without the projection
operation has accuracy that is almost identical to that of the standard,
fully-integrated DG method for this test problem. This finding is consistent
with the results shown in \cite{Persson2013}. However, when the projection
operation is performed in order to ensure discrete entropy stability, we observe
a larger error by approximately a constant factor. It is interesting to note
that for this test problem, we do not observe the sub-optimal order of accuracy
seen in previous test cases. This is possibly due to the translational nature of
the true solution.

\begin{figure}
  \includegraphics{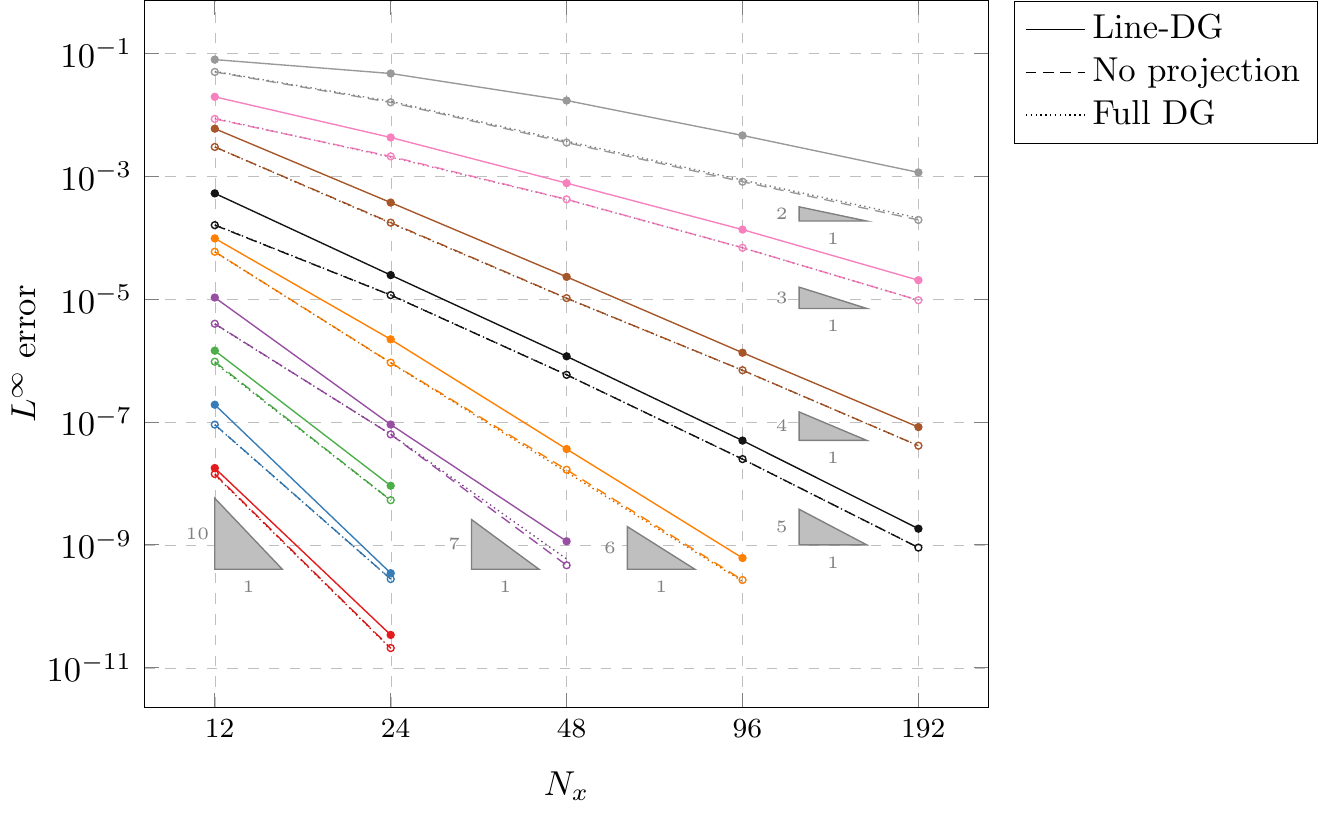}
  \caption{$L^\infty$ error for isentropic Euler vortex problem with polynomial
           degrees $p=1,2,\ldots,9$. Optimal order of accuracy is observed
           for all cases. Solid lines indicate the entropy-stable Line-DG
           method. Dashed lines indicate Line-DG without the projection
           operation. Dotted lines indicate fully-integrated standard DG.}
  \label{fig:everrors}
\end{figure}

\subsection{2D supersonic flow in a duct}

We consider the inviscid supersonic flow in a duct with a smooth bump. The
duct has dimensions $3 \times 1$ and the $y$-coordinate of the bottom boundary
is given by
\begin{equation}
  y(x) = \begin{cases}
    \frac{H}{2}\left(\cos(\pi(1+2(x-3/4)) + 1\right), &\qquad 3/4 < x < 5/4 \\
    H, &\qquad 5/4 \leq x \leq 7/4 \\
    \frac{H}{2}\left(\cos(2\pi(x-7/4))+1\right), &\qquad 7/4 < x < 9/4 \\
    0, &\qquad\text{otherwise.}
  \end{cases}
\end{equation}
where the height of the bump is given by $H = 0.04$. The inflow density is $\rho
= 1$, and the inflow velocity is $\bm w = (1, 0)$. We set the Mach number to $M
= 1.4$ as in \cite{Fernandez2018}. Slip wall conditions are enforced on the top
and bottom boundaries. The curved boundary is represented using isoparametric
elements. We use 675 elements with degree $p=4$ polynomials. We do not use any
shock capturing techniques or apply any limiters to the solution. This test case
is intended to assess the robustness of the method for under-resolved high Mach
number flow.

We compute the steady solution to this problem using pseudo-time integration. We
compare the solutions obtained using the $\mu = p+1$ DG-SEM method to the
Line-DG method with $\mu > p+1$. The steady-state pressure is shown in Figure
\ref{fig:duct}. Both solutions display fairly severe oscillations in the
vicinity of the shocks. Some of these features appear to be more prominent in
the solution obtained using the DG-SEM method. Despite these oscillations, the
method remains robust due to the entropy stability. Traditional DG-SEM, Line-DG,
or consistently-integrated standard DG methods are unstable for this problem
without the use of additional shock capturing techniques or limiters.

\begin{figure}
\includegraphics{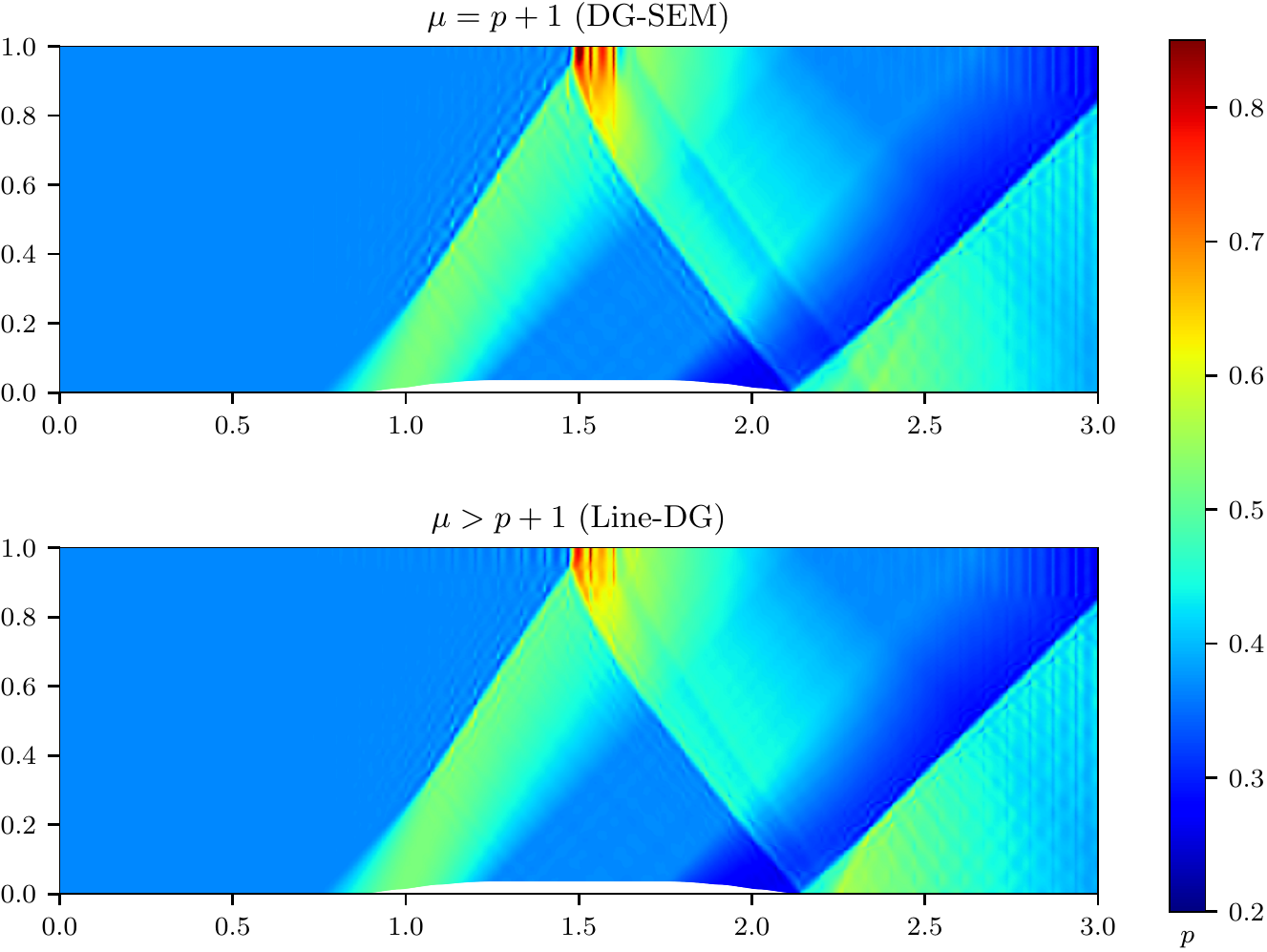}
\caption{Density of steady-state solution to supersonic flow over a bump.}
\label{fig:duct}
\end{figure}

We also use this test to verify the freestream preservation of the
method, which is widely known to be an important property \cite{Thomas1979}. We
initialize the solution to uniform flow and enforce freestream boundary
conditions at all boundaries. We then integrate in time and monitor the
deviation of the solution from the initial freestream conditions. We use $p=2$
polynomials and consider both the Line-DG and DG-SEM methods, and integrate
until a final time of $t=1$. Both methods resulted in only machine precision
deviation from freestream conditions. The $L^\infty$ deviation from the
freestream density was found to be about $4.4 \times 10^{-14}$ using the DG-SEM
method and about $4.9 \times 10^{-14}$ using the Line-DG method.

\subsection{3D inviscid Taylor-Green vortex}

For a final set of test cases, we consider the compressible, inviscid
Taylor-Green vortex (TGV) \cite{Taylor1937} at different Mach numbers. This
problem has been extensively studied for the incompressible case
\cite{Brachet1983}, as well as the nearly-incompressible case
\cite{Shu2005,Chapelier2012,Pazner2017b,Pazner2018}. The stability of DG
discretizations for the under-resolved simulation of the inviscid TGV has also
been studied in \cite{Winters2018,Moura2017}. The domain is taken to be the cube
$[-\pi,\pi]^3$, and periodic conditions are enforced on all boundaries. The
initial conditions are given by
\begin{equation}
\begin{aligned}
  \rho(x,y,z) &= \rho_0 \\
  w_1(x,y,z) &= w_0 \sin(x) \cos(y) \cos(z) \\
  w_2(x,y,z) &= -w_0 \cos(x) \sin(y) \cos(z) \\
  w_3(x,y,z) &= 0 \\
  p(x,y,z) &= p_0 + \rho_0 u_0^2\left(
    \cos(2x) + \cos(2y) \right)\left(\cos(2z) + 2\right)/16,
\end{aligned}
\end{equation}
where we take the parameters to be $w_0 = 1$, $\rho_0 = 1$, with Mach number
$M_0 = u_0/c_0$, where $c_0$ is the speed of sound computed in accordance with
the pressure $p_0$. The characteristic convective time is given by $t_{\rm
c}=1$, and we integrate until $t = 10 t_{\rm c}$. For the nearly incompressible
case, we choose $p_0 = 100$ which corresponds to a Mach number of $M_0 \approx
0.08$. We also consider a higher Mach number case defined by $M_0 = 0.7$.

We measure three quantities of interest. The first is the mean entropy, which is
guaranteed to be monotonically non-increasing by the method. The second is the
mean kinetic energy
\begin{equation}
  E_k(t) = \frac{1}{\rho_0 |\Omega|} \int_\Omega \tfrac{1}{2} \rho \bm w\cdot
  \bm w \, d\bm x.
\end{equation}
We can easily see that $E_k(0) = 1/8$. Since the kinetic energy is conserved for
the inviscid Taylor-Green vortex in the incompressible limit, for the low Mach
(nearly incompressible) case, we can use $E_k(t)$ as a measure of the numerical
dissipation introduced by the discretization. The third quantity of interest
considered is the mean enstrophy, defined by
\begin{equation}
  \E(t) = \frac{1}{\rho_0 |\Omega|} \int_\Omega \tfrac{1}{2} \rho
    \bm\omega\cdot\bm\omega \, d\bm{x},
\end{equation}
where $\bm\omega = \nabla \times \bm w$ is the vorticity. The enstrophy can be
used as a measure of the resolving power of the numerical discretization
\cite{Shu2005}. These integrals are discretized using a consistent quadrature
with $\mu$ points in each dimension.

We discretize the geometry using a $20 \times 20 \times 20$ Cartesian grid, and
we use degree $p=3$ and $p=5$ polynomials. For the higher Mach number case,
defined by $M=0.7$, the standard DG-SEM method without entropy stability is
unstable after about $t = 3.9 t_{\rm c}$. The Line-DG method without
entropy stability is unstable after about $t = 4.7 t_{\rm c}$. The
entropy-stable versions of both DG-SEM and Line-DG remain stable for the full
duration of the simulation. These stability issues were not
observed for the nearly incompressible case.

In Figure \ref{fig:tgquantities}, we show the normalized time evolution of the
quantities of interest for both test cases. We define the normalized mean
entropy by $\int_\Omega \left( U(\bm x, t) - U(\bm x, 0) \right) \, d\bm x /
\int_\Omega U(\bm x, 0) \, d\bm x$, and similarly for the normalized mean
kinetic energy and enstrophy. For the low Mach case with $p=3$ polynomials, we
notice that the Line-DG method dissipates less entropy and kinetic energy than
the equal-order DG-SEM method. In fact, the dissipation of these two quantities
is roughly equivalent to the DG-SEM method with polynomial degree $p=5$. For
both the Line-DG method and DG-SEM method with $p=3$, the peak enstrophy is
under-predicted. For the $M_0=0.7$ case, the Line-DG and DG-SEM methods result
in comparable results for all three quantities considered, however for the $p=5$
case, the DG-SEM method gives rise to less enstrophy growth. As discussed in
\cite{Shu2005}, some caution is required when using these mean quantities to
assess the quality of the numerical solutions, in particular once the solution
has become under-resolved.

This test case demonstrates both the increased robustness of the entropy-stable
Line-DG when compared with standard DG methods, and its low dissipation when
compared with the equal-order entropy-stable DG-SEM method.

\begin{figure}
\includegraphics{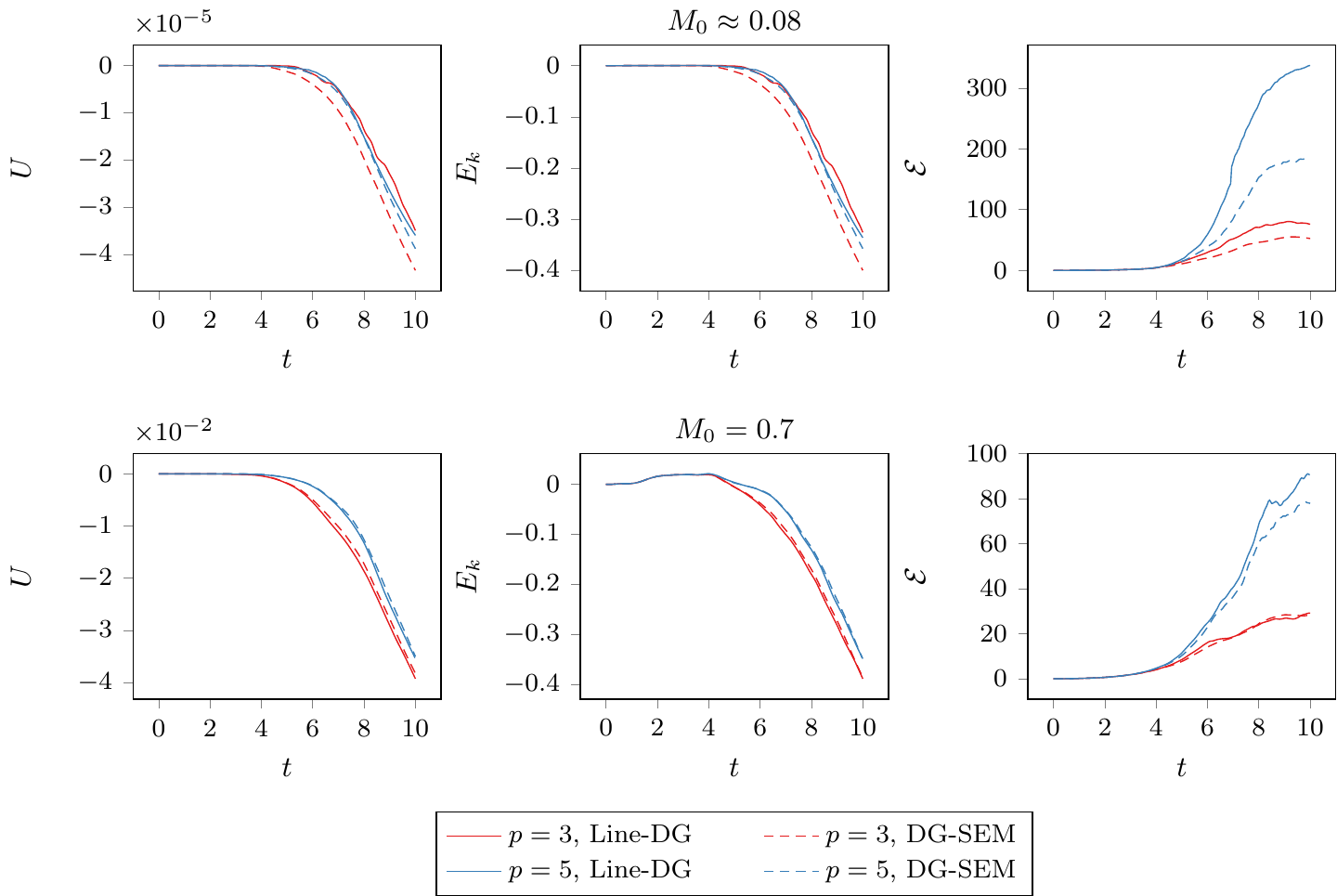}
\caption{Normalized time evolution of $U$ (entropy), $E_k$ (kinetic energy), and
         $\mathcal{E}$ (enstrophy) over time, for the 3D inviscid Taylor-Green
         test case, at $M_0\approx0.08$ and $M_0=0.7$.}
\label{fig:tgquantities}
\end{figure}

\section{Conclusions} \label{sec:conclusion}

In this paper, we have constructed a discretely entropy-stable line-based
discontinuous Galerkin method. We modify the Line-DG method of
\cite{Persson2012,Persson2013} using a flux differencing technique in order to
obtain discrete entropy stability, compatible with the quadrature rule used in
the discretization. This line-based method is composed of one-dimensional
operations performed along lines or curves of nodes within tensor-product
elements, resulting in fewer flux evaluations, and requiring only
one-dimensional interpolation and integration operations. This method is closely
related to the entropy-stable DG-SEM method, described in
\cite{Chen2017,Carpenter2014,Fisher2013}, and to the entropy-stable full DG
method developed in \cite{Chan2018}. When compared with the equal-order
entropy-stable DG-SEM method on a range of test cases, the Line-DG method
results in smaller errors and less numerical dissipation.

The main feature of this entropy-stable method is its increased robustness in
the presence of shocks or under-resolved features. This robustness has been
demonstrated on a range of problems for which standard DG-type methods are
unstable without the use of additional limiting or shock-capturing techniques.
For problems with strong shocks, the entropy-stable Line-DG method can
demonstrate spurious oscillations despite remaining stable. For problems of this
type, artificial viscosity or limiters may be used to reduce the oscillations
and increase solution quality.

\section{Acknowledgements}

This work was supported by the National Aeronautics and Space Administration
(NASA) under grant number NNX16AP15A, by the Director, Office of Science, Office
of Advanced Scientific Computing Research, U.S. Department of Energy under
Contract No. DE-AC02-05CH11231 and by the AFOSR Computational Mathematics
program under grant number FA9550-15-1-0010. Lawrence Livermore National
Laboratory is operated by Lawrence Livermore National Security, LLC, for the
U.S. Department of Energy, National Nuclear Security Administration under
Contract DE-AC52-07NA27344 (LLNL-JRNL-767379).

\bibliographystyle{spmpsci}
\bibliography{refs2}

\end{document}